\documentclass[10pt,a4paper,reqno,twoside]{amsart}

\usepackage{verbatim}
\usepackage{amssymb}
\usepackage{upref}
\usepackage{enumerate}
\usepackage{latexsym, amsfonts, amsmath, amscd, esint}
\usepackage{mathrsfs}
\usepackage{amssymb}

\makeatletter

\theoremstyle{plain}
\newtheorem{theorem}{Theorem}

\newtheorem{proposition}{Proposition}[section]
\newtheorem{lemma}{Lemma}[section]
\newtheorem{corollary}{Corollary}[section]
\theoremstyle{definition}
\newtheorem{definition}{Definition}
\newtheorem{hypothesis}{Hypothesis}

\newtheorem*{Not}{Notation}
\theoremstyle{remark}
\newtheorem{remark}{Remark}[section]
\newtheorem{example}[remark]{Example}

\allowdisplaybreaks

\newcommand{\N}{\mathbb{N}}

\newcommand{\R}{\mathbb{R}}

\newcommand{\xii}{{|\xi|}}

\newcommand{\lloc}{{\mathrm{loc}}}

\renewcommand{\div}{{\mathrm{\,div\,}}}
\newcommand{\curl}{{\mathrm{\,curl\,}}}

\def\<#1\>{\left\langle#1\right\rangle }

\title[Hyperbolic equations with dissipative lower order terms]
{Asymptotics of higher order hyperbolic equations \\ with one or two dissipative lower order terms}
\author{Marcello D'Abbicco}

\address{Marcello D'Abbicco, Dept. of Mathematics, University of Bari, Via E. Orabona 4 - 70125 BARI - ITALY}

\begin{document}

\begin{abstract}
In this paper, we consider the Cauchy problem for a hyperbolic equation~$Q(\partial_t,\partial_x)u=0$ of any order $m\geq3$, where~$t\geq0$ and~$x\in\R^n$, and $Q=P_m+P_{m-1}+P_{m-2}$ is a sum of homogeneous hyperbolic polynomials $P_{m-j}$ of order $m-j$. We assume the sufficient and necessary condition which guarantees the strict stability of the polynomial~$Q(\lambda,i\xi)$, for any~$\xi\neq0$. Under this assumption, we derive a polynomial decay rate for the energy of the problem, in different scenarios of interlacing of the polynomials $P_{m-j}(\lambda,\xi)$, and we describe the asymptotic profile of the solution as $t\to\infty$, assuming a moment condition on the initial data. In order to do this, we study the asymptotic behavior of the $m$ roots of the full symbol $Q(\lambda,i\xi)=0$, as~$\xi\to0$ and as $\xii\to\infty$. Examples of models to which the results may be applied include the theory of acoustic waves and the theory of electromagnetic elastic waves. Also, as an application, we prove the existence of global small data solutions to the problem with supercritical power nonlinearities of type $|D^\alpha u|^p$, with $|\alpha|\leq m-2$.
\end{abstract}

\keywords{hyperbolic polynomials, higher order dissipative equations, asymptotic profile, acoustic waves, elastic waves, critical exponent}

\subjclass[2020]{35L25, 35L30, 35B40, 35L76, 35B33}

\maketitle

\section{Introduction}\label{sec:intro}

\baselineskip13pt

In this paper, we consider the forward Cauchy problem
\begin{equation}\label{eq:CPlin}
\begin{cases}
Q(\partial_t,\partial_x) u =0, \quad t\geq 0, x\in\R^n,\\
\partial_t^j u(0,x)=u_j(x),\qquad j=0,\ldots,m-1,\\
\end{cases}
\end{equation}
where $Q(\partial_t,\partial_x)$ is a hyperbolic operator of order~$m$, obtained by the sum of two or three homogeneous hyperbolic operators of order~$m$ and~$m-1$ or, respectively, $m$, $m-1$, $m-2$. We choose lower order polynomials which induce a partial dissipation of the energy for~\eqref{eq:CPlin}, that is, the roots of the full symbol $Q(\lambda,i\xi)$ have negative real part for any~$\xi\neq0$, but some roots vanish as $\xi\to0$. Under this assumption, we describe the asymptotic behavior of those roots as~$\xi\to0$ and as~$\xii\to\infty$. In particular, we study the vanishing speed of the real parts of those roots as $\xi\to0$ (and, in some cases, as $\xii\to\infty$). With this information, we obtain the decay rate for energy estimates for the solution to~\eqref{eq:CPlin} (Theorems~\ref{thm:CPQ2} and~\ref{thm:verygen}) and we describe the asymptotic profile of the solution, assuming a moment condition on the initial data (Theorem~\ref{thm:asymp}).
\begin{definition}\label{def:hyperbolic}
A complex polynomial $p(z)$ is hyperbolic if its roots are real-valued. It is strictly hyperbolic if its roots are real and simple. A complex polynomial $q(z)$ is strictly stable if $\Re z<0$ for any root of $q(z)$.
\end{definition}
\begin{definition}\label{def:stable}
A complex polynomial $q(z)$ is strictly stable if $\Re z<0$ for any root of $q(z)$.
\end{definition}
\begin{definition}\label{def:interlace}[see Definition 6.3.1 in~\cite{R-S}]
Let $p_{m-1}(z)$ be a hyperbolic polynomial of degree~$m-1$ with roots $\tilde\lambda_1\leq\ldots\leq\tilde\lambda_{m-1}$, and let $p_m(z)$ be a hyperbolic polynomial of degree~$m$ with roots $\lambda_1\leq\ldots\leq\lambda_m$.

We say that $p_{m-1}(z)$ and $p_m(z)$ interlace if
\begin{equation}\label{eq:weakinterlace}
\lambda_1 \leq \tilde\lambda_1 \leq \lambda_2 \leq \tilde\lambda_2 \leq \ldots \leq \tilde\lambda_{m-1}\leq \lambda_m.
\end{equation}
We say that $p_{m-1}(z)$ and $p_m(z)$ strictly interlace if
\begin{equation}\label{eq:interlace}
\lambda_1 < \tilde\lambda_1 < \lambda_2 < \tilde\lambda_2 < \ldots < \tilde\lambda_{m-1}< \lambda_{m}.
\end{equation}
\end{definition}
We stress that if two polynomials strictly interlace, then they are both strictly hyperbolic.
\begin{definition}\label{def:hyp}
Let
\begin{equation}\label{eq:P}
P_m(\partial_t,\partial_x)=\sum_{k+|\alpha|=m} c_{k,\alpha} \partial_t^k\partial_x^\alpha,
\end{equation}
be a homogeneous operator of order~$m$, with~$c_{m,0}>0$. We say that~$P(\partial_t,\partial_x)$ is hyperbolic if its symbol
\[ P(\lambda,\xi)=\sum_{k+|\alpha|=m} c_{k,\alpha} \lambda^k\xi^\alpha, \]
is hyperbolic, that is, it admits only real-valued roots~$\lambda_j(\xi')$, for any given~$\xi'\in S^{n-1}=\{\xi\in\R^n: \ \xii=1\}$. We say that~$P(\partial_t,\partial_x)$ is strictly hyperbolic if $P(\lambda,\xi')$ is strictly hyperbolic, that is, the roots~$\lambda_j(\xi')$ are distinct for any given~$\xi'\in S^{n-1}$.
\end{definition}
We choose $Q(\partial_t,\partial_x)=Q_1(\partial_t,\partial_x)$ or $Q(\partial_t,\partial_x)=Q_2(\partial_t,\partial_x)$, in~\eqref{eq:CPlin}, where
\begin{align}
\label{eq:Q1}
Q_1(\partial_t,\partial_x)
     & = P_m(\partial_t,\partial_x)+P_{m-1}(\partial_t,\partial_x),\\
\label{eq:Q2}
Q_2(\partial_t,\partial_x) & = Q_1(\partial_t,\partial_x)+P_{m-2}(\partial_t,\partial_x),
\end{align}
with $P_{m-j}(\partial_t,\partial_x)$ hyperbolic operators of order~$m-j$. Without loss of generality, we assume in the following that~$c_{m,0}=1$ and $c_{m-1,0}>0$; moreover, $c_{m-2,0}>0$ if we consider~$Q_2$.


By the Hermite--Biehler theorem  (see, for instance, \cite[Theorem 6.3.4]{R-S}), for any fixed~$\xi\neq0$, the strict stability of $Q_1(\lambda,i\xi)$ or $Q_2(\lambda,i\xi)$ is equivalent to the strict interlacing of $P_{m-1}(\lambda,\xi)$ and $P_m(\lambda,\xi)$, or, respectively, $P_m(\lambda,\xi)-P_{m-2}(\lambda,\xi)$. It is natural to write this condition for $Q_1(\lambda,i\xi)$ with $\xi\neq0$, by means of a condition on the roots of $P_{m-1}(\lambda,\xi')$ and~$P_m(\lambda,\xi')$, with $\xi'\in S^{n-1}$.
\begin{lemma}\label{lem:stable1}(see~\cite{VD}, see also~\cite[Proposition 2.2]{VV})
The polynomial $Q_1(\lambda,i\xi)$ is strictly stable for any $\xi\neq0$ if, and only if, $P_m(\partial_t,\partial_x)$ and $P_{m-1}(\partial_t,\partial_x)$ are strictly hyperbolic operators, and $P_{m-1}(\lambda,\xi')$ and~$P_m(\lambda,\xi')$ strictly interlace, as in~\eqref{eq:interlace}, for any~$\xi'\in S^{n-1}$.
\end{lemma}
%
For $Q=Q_2$, the situation is more complicated.
\begin{hypothesis}\label{hyp:Q2}
We make the following assumption:
\begin{itemize}
\item $P_m(\partial_t,\partial_x)$ and~$P_{m-2}(\partial_t,\partial_x)$ are hyperbolic;
\item $P_{m-1}(\partial_t,\partial_x)$ is strictly hyperbolic;
\item $P_{m-1}(\lambda,\xi')$ and~$P_m(\lambda,\xi')$ interlace, as in~\eqref{eq:weakinterlace}, for any~$\xi'\in S^{n-1}$;
\item $P_{m-2}(\lambda,\xi')$ and~$P_{m-1}(\lambda,\xi')$ interlace, as in~\eqref{eq:weakinterlace}, for any~$\xi'\in S^{n-1}$;
\item there is no $(\lambda,\xi')\in\R\times S^{n-1}$ such that $P_m(\lambda,\xi')=P_{m-1}(\lambda,\xi')=P_{m-2}(\lambda,\xi')=0$.
\end{itemize}
\end{hypothesis}
\begin{lemma}\label{lem:stable2}(\cite[Theorem 2.1]{VV})
The polynomial $Q_2(\lambda,i\xi)$ is strictly stable if, and only if, Hypothesis~\ref{hyp:Q2} holds.
\end{lemma}
The strict stability of the polynomial $Q(\lambda,i\xi)$ for any $\xi\neq0$, with $Q=Q_1$ or $Q=Q_2$, guaranteed by Lemmas~\ref{lem:stable1} and~\ref{lem:stable2}, respectively, means that~$\Re \lambda_j(\xi)<0$ for any~$\xi\neq0$ and~$j=1,\ldots,m$, where $\lambda_j(\xi)$ are the roots of $Q(\lambda,i\xi)$. Since the roots of a polynomial continuously depend on its coefficients (see, for instance, \cite[Theorem 1.31]{R-S}), we get a uniform bound~$\Re\lambda_j(\xi)\leq -c$ in any compact subset~$K$ of~$\R^n\setminus\{0\}$, for some~$c=c(K)>0$. This property is sufficient to obtain exponential decay in time of $\hat u(t,\xi)$ in such compact subset of~$\R^n\setminus\{0\}$. To obtain decay estimates and to describe the asymptotic profile of the solution to~\eqref{eq:CPlin}, we need to study the asymptotic behavior of $\lambda_j(\xi)$ as $\xi\to0$ and as $\xii\to\infty$.

The strict interlacing of $P_{m-1}(\lambda,\xi')$ and~$P_m(\lambda,\xi')$ for any $\xi'\in S^{n-1}$ is sufficient to obtain an easy asymptotic behavior of the roots of $Q_1(\lambda,i\xi)$ as $\xi\to0$ and as $\xii\to\infty$, see later, respectively, \eqref{eq:strictlow} and~\eqref{eq:stricthigh}. The situation is more complicated for $Q=Q_2$ under the assumption of Hypothesis~\ref{hyp:Q2}, but it is much simpler if we strengthen this assumption requiring the strict interlacing of $P_{m-1}(\lambda,\xi')$ and~$P_m(\lambda,\xi')$, and of $P_{m-2}(\lambda,\xi')$ and~$P_{m-1}(\lambda,\xi')$, for any $\xi'\in S^{n-1}$.
\begin{theorem}\label{thm:CPQ2}
Assume that~$P_m(\partial_t,\partial_x)$ and~$P_{m-1}(\partial_t,\partial_x)$ are strictly hyperbolic, that the polynomials $P_{m-1}(\lambda,\xi')$ and $P_m(\lambda,\xi')$ strictly interlace, for any~$\xi'\in S^{n-1}$, and that
\[ u_j\in H^{s_0-j} \cap L^q \qquad \text{for some $s_0\geq m-1$, $q\in[1,2]$, for $j=0,\ldots,m-1$.} \]
Then the solution $u$ to Cauchy problem~\eqref{eq:CPlin} with $Q=Q_1$ satisfies the following decay estimate
\begin{equation}
\label{eq:estQ1}\begin{split}
\|\partial_t^ku(t,\cdot)\|_{\dot H^s}
    & \leq C\,(1+t)^{-\frac{n}2\left(\frac1q-\frac12\right)-\frac{k+s-(m-2)}2}\,\sum_{j=m-2,m-1}\big(\|u_j\|_{L^q}+\|u_j\|_{H^{k+s-j}}\big) \\
    & \qquad + C\,\sum_{j=0}^{m-3}(1+t)^{-\frac{n}2\left(\frac1q-\frac12\right)-\frac{k+s-j}2}\,\big(\|u_j\|_{L^q}+\|u_j\|_{H^{k+s-j}}\big),
\end{split}\end{equation}
for any $k+s\leq s_0$, such that $k+s\geq m-2$ if $q=2$, or $n(1/q-1/2)+k+s>m-2$, otherwise, with~$C>0$ independent of the initial data.

Assume, moreover, that~$P_{m-2}(\partial_t,\partial_x)$ is strictly hyperbolic and that the polynomials~$P_{m-2}(\lambda,\xi')$ and~$P_{m-1}(\lambda,\xi')$ strictly interlace, for any~$\xi'\in S^{n-1}$. Then the solution to Cauchy problem~\eqref{eq:CPlin} with $Q=Q_2$ satisfies the following decay estimate
\begin{equation}
\label{eq:estQ2}\begin{split}
\|\partial_t^k u(t,\cdot)\|_{\dot H^s}
    & \leq C\,(1+t)^{-\frac{n}2\left(\frac1q-\frac12\right)-\frac{k+s-(m-3)}2}\,\sum_{j=m-3}^{m-1}\big(\|u_j\|_{L^q}+\|u_j\|_{H^{k+s-j}}\big)\\
     & \qquad + C\,\sum_{j=0}^{m-4}(1+t)^{-\frac{n}2\left(\frac1q-\frac12\right)-\frac{k+s-j}2}\,\big(\|u_j\|_{L^q}+\|u_j\|_{H^{k+s-j}}\big),
\end{split}
\end{equation}
for $k+s\leq s_0$, such that $k+s\geq m-3$ if $q=2$, or $n(1/q-1/2)+k+s>m-3$, otherwise, with~$C>0$ independent of the initial data.
\end{theorem}
We stress that in estimate~\eqref{eq:estQ2} an additional power $(1+t)^{-\frac12}$ appears, with respect to estimate~\eqref{eq:estQ1}. This extra decay rate is due to the presence of the additional dissipative term $P_{m-2}(\partial_t,\partial_x)$.

It is well-known that problem~\eqref{eq:CPlin} is $H^s$ well-posed, since $P_m(\partial_t,\partial_x)$ is strictly hyperbolic and has constant coefficients. The decay rate which appears taking higher order derivatives and assuming initial data in $L^q$, for some $q\in[1,2)$, is a consequence of the presence of the dissipative terms $P_{m-1}(\partial_t,\partial_x)$ and $P_{m-2}(\partial_t,\partial_x)$, and it is a typical phenomenon of parabolic problems.
\begin{remark}
Theorem~\ref{thm:CPQ2} for $Q=Q_1$ naturally extends results which are very well known for damped wave equations~\cite{M76} and analogous second-order equations; indeed, we may formally set $m=2$ and Theorem~\ref{thm:CPQ2}, as well as the other results of this paper (Lemmas~\ref{lem:low} and~\ref{lem:high}, Theorem~\ref{thm:asymp}), remains valid for $Q=Q_1$. However, in that case, the roots of $Q_1(\lambda,i\xi)$ may be explicitly computed, so there is no benefit from our analysis.
\end{remark}
When $q=1$, the optimality of the decay rate 
in~\eqref{eq:estQ1} and of the decay rate 
in~\eqref{eq:estQ2} follows as a consequence of the asymptotic profile of the solution to~\eqref{eq:CPlin} (see later, Theorem~\ref{thm:asymp}), under the moment condition $M\neq0$, where
\begin{equation}\label{eq:moment}
M = \int_{\R^n} \big( u_{m-1}(x) + c_{m-1,0}\, u_{m-2}(x) + c_{m-2,0}\,u_{m-3}(x)\big)\,dx\,,
\end{equation}
where we set $c_{m-2,0}=0$ if we consider $Q=Q_1$. If $M=0$, then
\[ \|\partial_t^k u(t,\cdot)\|_{\dot H^s}=\textit{o}\big((1+t)^{-\frac{n}2\left(\frac1q-\frac12\right)-\frac{k+s-(m-3)}2}\big),\]
see Theorem~\ref{thm:asymp}.
\begin{remark}
In the case $Q=Q_1$, Theorem~\ref{thm:CPQ2} has been proved in~\cite{DAJ16} in the special case $P_{m-1}(\lambda,\xi) = c\,\partial_\lambda P_m(\lambda,\xi)$ for $c>0$. We emphasize that a strictly hyperbolic polynomial and its derivative always strictly interlace. In the case $Q=Q_2$, Theorem~\ref{thm:CPQ2} has been proved in~\cite{DAJ17} in the special case $P_{m-j}(\lambda,\xi) = c_j\,\partial_\lambda^j P_m(\lambda,\xi)$ for $c_j>0$, $j=1,2$. In~\cite[Theorem 1]{DAJ17}, also additional lower order terms are considered, under the assumption that the coefficients $\{c_j\}$ themselves are the coefficients of a strictly stable polynomial $\sum_{j=0}^{r} c_j z^{r-j}$. Moreover, weakly hyperbolic operators~$P_m$ are considered; when $Q=Q_2$, $P_m$ may have at most double roots. This scenario, studied in~\cite{DAJ17} for the special case $P_{m-j}(\lambda,\xi) = c_j\,\partial_\lambda^j P_m(\lambda,\xi)$, is generalized in \textsection~\ref{sec:weak} to more general interlacing assumptions for $P_m$, $P_{m-1}$, $P_{m-2}$, see Lemma~\ref{lem:highweak}, in particular.
\end{remark}
We stress that Hypothesis~\ref{hyp:Q2} is a weaker assumption than the assumption that $P_{m-1}(\lambda,\xi')$ and~$P_m(\lambda,\xi')$ strictly interlace, and $P_{m-2}(\lambda,\xi')$ and~$P_{m-1}(\lambda,\xi')$ strictly interlace, for any~$\xi'\in S^{n-1}$. However, under this stronger assumption, we have a simpler description of the asymptotic behavior of the roots $\lambda_j(\xi)$ of $Q_2(\lambda,i\xi)$, as $\xi\to0$ and as $\xii\to\infty$ (see later, \eqref{eq:strictlow}, \eqref{eq:lambdalow2}  and~\eqref{eq:lambdahighsimple},\eqref{eq:stricthigh}). This simpler behavior is also better, in the sense that it produces better decay estimates for the solution to~\eqref{eq:CPlin}. For this reason, in Theorem~\ref{thm:CPQ2}, we derived estimates under the stronger assumption of strict interlacing polynomials.

In \textsection~\ref{sec:weak} we discuss the different, interesting, scenarios which may occur when $Q(\lambda,i\xi)$ is strictly stable for any~$\xi\neq0$, 
i.e., Hypothesis~\ref{hyp:Q2} holds, but the interlacing of the polynomials is not strict. Four phenomena may appear, according to the ways in which the strict interlacing condition is weakened:
\begin{itemize}
\item the polynomial decay rate is related to a different scaling of the time and space variables if the interlacing of $P_{m-2}(\lambda,\xi')$ and $P_{m-1}(\lambda,\xi')$ is not strict for some $\xi'\in S^{n-1}$ (see~\eqref{en:strictlowloss} in Theorem~\ref{thm:lowloss});
\item a loss of decay rate appears in the decay estimate, if $P_{m-2}(\partial_t,\partial_x)$ is not strictly hyperbolic (see~\eqref{en:weaklowloss} in Theorem~\ref{thm:lowloss});
\item a \emph{regularity-loss type decay} appears if the interlacing of $P_{m-1}(\lambda,\xi')$ and $P_{m}(\lambda,\xi')$ is not strict for some $\xi'\in S^{n-1}$ (see Proposition~\ref{prop:lossdecay});
\item a loss of regularity appears if $P_m(\partial_t,\partial_x)$ is not strictly hyperbolic (see Proposition~\ref{prop:regloss}).
\end{itemize}
Assuming only Hypothesis~\ref{hyp:Q2}, i.e., the sufficient and necessary condition for the strict stability of $Q_2(\lambda,i\xi)$, all or some of the four phenomena may appear. On the one hand, the obtained decay estimate is weaker than~\eqref{eq:estQ2}, which holds if we assume the strict interlacing of $P_{m-2}(\lambda,\xi')$ and $P_{m-1}(\lambda,\xi')$ for any $\xi'\in S^{n-1}$. On the other hand, if $P_{m-1}(\lambda,\xi')$ and $P_{m}(\lambda,\xi')$  do not strictly interlace for any $\xi'\in S^{n-1}$, we shall require more initial data regularity. In particular, if $P_m(\partial_t,\partial_x)$ is hyperbolic, but not strictly, then~\eqref{eq:CPlin} is not well-posed in $H^s$, in general, but only in $H^\infty$, 
since a loss of $1$ derivative may occur.
\begin{theorem}\label{thm:verygen}
Assume that Hypothesis~\ref{hyp:Q2} holds, and that
\[ u_j\in L^q \cap H^\infty \qquad \text{for some $q\in[1,2]$, for $j=0,\ldots,m-1$.} \]
Then the solution $u$ to Cauchy problem~\eqref{eq:CPlin} with $Q=Q_2$ satisfies the following decay estimate
\begin{equation}
\label{eq:estQ2verygen}\begin{split}
\|\partial_t^ku(t,\cdot)\|_{\dot H^s}
    & \leq C\,\sum_{j=m-3}^{m-1}(1+t)^{-\eta}\,\|u_j\|_{L^q} + C\,\sum_{j=0}^{m-1} (1+t)^{-\frac\nu2}\,\|u_j\|_{H^{k+s+\nu-j}}\\
    & \qquad + C\,\sum_{j=0}^{m-4}(1+t)^{-\frac{n}4\left(\frac1q-\frac12\right)-\frac{k+s-j}4}\,\|u_j\|_{L^q},
\end{split}\end{equation}
for any~$\nu\geq1$, where
\begin{equation}
\label{eq:estQ2worst}
\eta=\min\left\{\frac{n}4\left(\frac1q-\frac12\right)+\frac{k+s-(m-3)}4, \ \frac{n}2\left(\frac1q-\frac12\right)+\frac{k+s-(m-2)}2\right\},
\end{equation}
provided that $k+s\geq m-3$ if $q=2$, or $n(1/q-1/2)+k+s>m-3$ otherwise, with~$C>0$ independent of the initial data.
\end{theorem}
The possible regularity loss of $1$ derivative is described by the requirement $\nu\geq1$ in Theorem~\ref{thm:verygen}; the regularity-loss type decay means that the decay rate $-\nu/2$ improves if additional regularity is taken on initial data. The decay rate in~\eqref{eq:estQ2worst} compares the influence of the possible source of loss of decay, with respect to~\eqref{eq:estQ2}. The possibility that $P_{m-2}(\partial_t,\partial_x)$ is not strictly hyperbolic leads to a loss of $(1+t)^{\frac12}$ power of decay rate, which become analogous to the decay rate in~\eqref{eq:estQ1}. On the other hand, the possibility that the interlacing of $P_{m-2}(\lambda,\xi')$ and $P_{m-1}(\lambda,\xi')$ is not strict, in general, leads to a decay rate structure, in which the power $4$ appears in the denominator of $n(1/q-1/2)+k+s-(m-3)$, instead of the power $2$. This phenomenon is related to a faster vanishing speed of the real part of some root of $Q_2(\lambda,i\xi)$, as $\xi\to0$.

\bigskip

In the setting of long-time decay estimates for higher order inhomogeneous equations with constant coefficients, we address the interested reader to~\cite{RS10}, where dispersive and Strichartz estimates are obtained. In particular, under different hypotheses of geometric type on the roots of the full symbol of the operator, the authors are able to derive $L^q-L^{q'}$ estimates, $q\in[2,\infty]$, where~$q'=q/(q-1)$, for inhomogeneous hyperbolic equations. The decay rate in these estimates has the classical form~$(1+t)^{-\kappa\left(\frac1q-\frac1{q'}\right)}$, where~$\kappa>0$ depends on the assumptions on the roots of the full symbol. 
Very weak dissipative effects have been considered for higher order hyperbolic equations with time-dependent coefficients in~\cite{DAE12}.

A huge literature exists for partially dissipative first-order hyperbolic systems with constant coefficients, under suitable assumptions on the term of order zero, and its relations with the first-order term. We address the interested reader to~\cite{SK85}, and to~\cite{BZ11} and the references therein, being aware that this cannot be an exhaustive list. Dissipative estimates for first-order hyperbolic systems with time-dependent coefficients have been obtained in~\cite{W14}.

An advantage of working with equations instead of systems is that one may possibly formulate more explicit assumptions to describe dissipative effects and investigate more in details the asymptotic profile of the solution to a problem.


\begin{Not}
We use the following notation:
\begin{itemize}
\item $\hat u(t,\xi)=\mathfrak{F}(u(t,\cdot))(\xi)$ denotes the Fourier transform of $u(t,x)$, with respect to the variable $x$;
\item for a given $\xi\neq0$, we put $\xi'=\xi/\xii$; $\xi'$ may also denote a point on the unit sphere $S^{n-1}=\{\xi\in\R^n: \ \xii=1\}$;
\item $L^q=L^q(\R^n)$, $q\in[1,\infty)$ denotes the usual space of function whose $q$-th power is integrable and $L^\infty=L^\infty(\R^n)$ denotes the usual space of essentially bounded functions;
\item $H^s$, $s\geq0$, denotes the Sobolev space of $L^2$ functions~$f$ such that $\xii^s\hat f$ is also in~$L^2$, with norm $\|f\|_{H^s}=\|(1+\xii^2)^{\frac{s}2}\hat f\|_{L^2}$; we define the quantity $\|f\|_{\dot H^s}=\|\xii^s\hat f\|_{L^2}$; for integer values of $s$, $H^s$ is the space of $L^2$ functions whose distributional derivatives of order not larger than $s$ are in $L^2$, and $\|f\|_{\dot H^s}\equiv\sum_{|\alpha|=s}\|\partial_x^\alpha f\|_{L^2}$; $H^\infty=\cap_{s>0} H^s$ is the subspace of $\mathcal C^\infty$ functions with all derivatives in $L^2$;
\item we write $f\lesssim g$ or $g\gtrsim f$ if there exists $C>0$ such that $f\leq Cg$; 
\item $I_a$ denotes the Riesz potential of order $a$, i.e., $\mathfrak{F}(I_af)=\xii^{-a}\hat f$ in appropriate sense.
\item in this paper, it is crucial to distinguish whether a polynomial is evaluated at $\xi$ (for instance, for the homogeneous polynomial $P_{m-j}(\lambda,\xi)$) or at $i\xi$ (for instance, for the inhomogeneous polynomial $Q(\lambda,i\xi)$); the variable of the roots of a polynomial will always be $\xi$, or $\xi'$ for homogeneous polynomials, independently if the polynomial is real or complex.
\end{itemize}
%
%
For any $\xi'\in S^{n-1}$, $a_1(\xi')$, \ldots, $a_m(\xi')$ denote the $m$ roots of~$P_m(\lambda,\xi')$, $b_1(\xi')$, \ldots, $b_{m-1}(\xi')$ denote the $m-1$ roots of~$P_{m-1}(\lambda,\xi')$ and $d_1(\xi')$, \ldots, $d_{m-2}(\xi')$ denote the $m-2$ roots of~$P_{m-2}(\lambda,\xi')$. Namely,
\begin{align*}
P_m(\lambda,\xi')
    & =\prod_{k=1}^m (\lambda-a_k(\xi')),\\
P_{m-1}(\lambda,\xi')
    & =c_{m-1,0}\,\prod_{k=1}^{m-1} (\lambda-b_k(\xi')),\\
P_{m-2}(\lambda,\xi')
    & =c_{m-2,0}\,\prod_{k=1}^{m-2} (\lambda-d_k(\xi')).
\end{align*}
On the other hand, $\lambda_1(\xi)$, \ldots, $\lambda_m(\xi)$ denote the $m$ complex-valued roots of $Q(\lambda,i\xi)$. These latter are not homogeneous, but to simplify some computation in the study of their asymptotic behavior, we introduce the auxiliary functions $\mu_j(\rho,\xi')=\rho^{-1}\lambda_j(\rho\xi')$, where $\rho=\xii$ and $\xi'=\xi/\xii$, for any $\xi\neq0$.

For a given homogeneous hyperbolic polynomial~$P_m(\lambda,\xi)$ of order $m$ with roots $\lambda_j(\xi')$ for any~$\xi'\in S^{n-1}$, we also introduce the following notation:
\begin{align*}
\check{P}_{m,j}(\lambda,\xi')
    & = \frac{P_m(\lambda,\xi')}{\lambda-\lambda_j(\xi')},\\
\check{P}_{m,j,\ell}(\lambda,\xi')
    & = \frac{P_m(\lambda,\xi')}{(\lambda-\lambda_j(\xi'))(\lambda-\lambda_\ell(\xi'))}.
\end{align*}
%
\end{Not}


\section{Asymptotic profiles of the roots of the symbol and of the solution to~\eqref{eq:CPlin}}\label{sec:asymp}

We may study the asymptotic behavior of the roots of $Q(\lambda,i\xi)$ as~$\xi\to0$ and as~$\xii\to\infty$ in a more general setting than the one in Theorem~\ref{thm:CPQ2}. In general, in the following, we put
\begin{equation}\label{eq:Q}
Q(\lambda,i\xi)= \sum_{j=0}^\ell P_{m-j}(\lambda,i\xi),
\end{equation}
where $P_{m-j}(\partial_t,\partial_x)$ are homogeneous operators of degree~$m-j$, with $c_{m,0}=1$ and $c_{m-\ell,0}>0$ (same notation as in~\eqref{eq:P}). Following as in~\cite{VV}, it is easy to see that a necessary condition for the strict stability of $Q(\lambda,i\xi)$ for any~$\xi\neq0$, is the interlacing of $P_{m-\ell}(\lambda,\xi')$ and $P_{m-\ell+1}(\lambda,\xi')$, and of $P_{m-1}(\lambda,\xi')$ and $P_m(\lambda,\xi')$, for any $\xi'\in S^{n-1}$.

Strengthening this assumption to the strict interlacing, the asymptotic behavior of the roots of $Q(\lambda,i\xi)$ is only determined by $P_{m-\ell}$ and $P_{m-\ell+1}$ as~$\xi\to0$, and by $P_m$ and $P_{m-1}$ as~$\xii\to\infty$, whereas the other polynomials come into play otherwise (see \textsection\ref{sec:weak}). The following two lemmas describe such behavior.
\begin{lemma}\label{lem:low}
Assume that~$P_{m-\ell}(\partial_t,\partial_x)$ and~$P_{m-\ell+1}(\partial_t,\partial_x)$ are strictly hyperbolic, and that~$P_{m-\ell}(\lambda,\xi')$ and~$P_{m-\ell+1}(\lambda,\xi')$ strictly interlace for any $\xi'\in S^{n-1}$. Then we may label the $m$ roots~$\lambda_j(\xi)$ of~$Q(\lambda,i\xi)=0$, $\xi\in\R^n$, in such a way that $\lambda_m(0)$, \ldots, $\lambda_{m-\ell+1}(0)$ are the solutions of
\begin{equation}\label{eq:lambda0}
\sum_{j=0}^{\ell} c_{m-j,0}\lambda^{\ell-j}=0,
\end{equation}
and
\begin{equation}\label{eq:lambdalow}
\lambda_j(\xi) = i\xii \beta_j(\xi') + \xii^2\,\frac{P_{m-\ell+1}(\beta_j(\xi'),\xi')}{\check{P}_{m-\ell,j}(\beta_j(\xi'),\xi')} + \textit{o}(\xii^2), \quad j=1,\ldots,m-\ell,
\end{equation}
as~$\xi\to0$, where~$\beta_j(\xi')$ are the $m-\ell$ real roots of~$P_{m-\ell}(\lambda,\xi')$, and~$\xi'=\xi/\xii\in S^{n-1}$.
\end{lemma}
The strict interlacing of~$P_{m-\ell}(\lambda,\xi')$ and~$P_{m-\ell+1}(\lambda,\xi')$ and the assumption that $c_{m-\ell},c_{m-\ell+1}>0$, guarantee that
\begin{equation}\label{eq:strictlow}
\frac{P_{m-\ell+1}(\beta_j(\xi'),\xi')}{\check{P}_{m-\ell,j}(\beta_j(\xi'),\xi')} = \frac{c_{m-\ell+1,0}\prod_{k=1}^{m-\ell+1} (\beta_j(\xi')-\alpha_k(\xi'))}{c_{m-\ell,0}\prod_{k\neq j}(\beta_j(\xi')-\beta_k(\xi'))} <0,
\end{equation}
in~\eqref{eq:lambdalow}, where $\alpha_j$ are the $m-\ell+1$ real roots of~$P_{m-\ell+1}(\lambda,\xi')$.

In particular, for~$\ell=1,2$ in Lemma~\ref{lem:low}, we obtain the following information:
\begin{itemize}
\item if~$Q=Q_1$ and~$P_{m-1}(\lambda,\xi')$ and~$P_m(\lambda,\xi')$ strictly interlace for any $\xi'\in S^{n-1}$, then $\lambda_m(0) = -c_{m-1,0}$ and
\begin{equation}\label{eq:lambdalow1}
\lambda_j(\xi) = i\xii b_j(\xi') + \xii^2\,\frac{P_{m-\ell+1}(b_j(\xi'),\xi')}{\check{P}_{m-\ell,j}(b_j(\xi'),\xi')} + \textit{o}(\xii^2), \quad j=1,\ldots,m-1,
\end{equation}
as~$\xi\to0$, where~$b_j(\xi')$ are the $m-1$ real roots of~$P_{m-1}(\lambda,\xi')$, and~$\xi'=\xi/\xii\in S^{n-1}$;
\item if~$Q=Q_2$ and~$P_{m-2}(\lambda,\xi')$ and~$P_{m-1}(\lambda,\xi')$ strictly interlace for any $\xi'\in S^{n-1}$, then
\begin{equation}\label{eq:lambdam-1m}
\lambda_{m-1,m}(0) = \begin{cases}
\dfrac{-c_{m-1,0}\pm \sqrt{c_{m-1,0}^2-4c_{m-2,0}}}2 & \text{if~$c_{m-1,0}^2\geq4c_{m-2,0}$,}\\
\dfrac{-c_{m-1,0}\pm i \sqrt{4c_{m-2,0}-c_{m-1,0}^2}}2 & \text{if~$c_{m-1,0}^2<4c_{m-2,0}$,}
\end{cases}\end{equation}
and
\begin{equation}\label{eq:lambdalow2}
\lambda_j (\xi) = i\xii d_j(\xi') + \xii^2\,\frac{P_{m-1}(d_j(\xi'),\xi')}{\check{P}_{m-2,j}(d_j(\xi'),\xi')} + \textit{o}(\xii^2),\quad j=1,\ldots,m-2,
\end{equation}
as~$\xi\to0$, where~$d_j(\xi')$ are the $m-2$ roots of~$P_{m-2}(\lambda,\xi')$, and~$\xi'=\xi/\xii\in S^{n-1}$.
\end{itemize}
%
%
%
\begin{lemma}\label{lem:high}
Assume that~$P_{m-1}(\partial_t,\partial_x)$ and~$P_{m}(\partial_t,\partial_x)$ are strictly hyperbolic, and that~$P_{m-1}(\lambda,\xi')$ and~$P_m(\lambda,\xi')$ strictly interlace for any $\xi'\in S^{n-1}$. Then we may label the $m$ roots~$\lambda_j(\xi)$ of~$Q(\lambda,i\xi)$, $\xi\in\R^n$, in such a way that:
\begin{equation}\label{eq:lambdahighsimple}
\lambda_j (\xi) = i\xii a_j(\xi') - \frac{P_{m-1}(a_j(\xi'),\xi')}{\check{P}_{m,j}(a_j(\xi'),\xi')} + \textit{o}(1),\quad j=1,\ldots,m,
\end{equation}
as~$\xii\to\infty$, where~$a_j(\xi')$ are the $m$ real roots of~$P_m(\lambda,\xi')$ and~$\xi'=\xi/\xii\in S^{n-1}$.
\end{lemma}
The strict interlacing of~$P_{m-1}(\lambda,\xi')$ and~$P_m(\lambda,\xi')$ guarantees that
\begin{equation}\label{eq:stricthigh}
- \frac{P_{m-1}(a_j(\xi'),\xi')}{\check{P}_{m,j}(a_j(\xi'),\xi')} = - \frac{c_{m-1,0}\prod_{k=1}^{m-1}(a_j(\xi')-b_k(\xi'))}{\prod_{k\neq j}(a_j(\xi')-a_k(\xi'))} <0,
\end{equation}
in~\eqref{eq:lambdahighsimple}, where~$b_j(\xi')$ are the $m-1$ roots of~$P_{m-1}(\lambda,\xi')$. 
%
%
\begin{remark}
In the special case $P_{m-j}(\lambda,\xi)=\partial_\lambda^jP_m(\lambda,\xi)$, a slightly different version of Lemmas~\ref{lem:low} and~\ref{lem:high} has been previously obtained in~\cite[Lemmas~2 and~3]{DAJ17}.
\end{remark}

\subsection*{Asymptotic profile of the solution to~\eqref{eq:CPlin}}

Thanks to Lemma~\ref{lem:low}, assuming initial data in $L^1$, we may refine Theorem~\ref{thm:CPQ2} to describe the asymptotic profile of the solution to~\eqref{eq:CPlin} as $t\to\infty$. We define $v(t,\cdot)$ and $w(t,\cdot)$ in $L^2$ by means of
\begin{align}
\label{eq:v}
\hat v(t,\xi)
    & = M\,\sum_{j=1}^{m-1}\frac{e^{i\xii\,b_j(\xi')t+\xii^2\,\frac{P_m(b_j(\xi'),\xi')}{\check{P}_{m-1,j}(b_j(\xi'),\xi')}\,t}}{i^{m-2}\,\check{P}_{m-1,j}(b_j(\xi'),\xi')}\,,\\
\label{eq:w}
\hat w(t,\xi)
    & = M\,\sum_{j=1}^{m-2}\frac{e^{i\xii\,d_j(\xi')t+\xii^2\,\frac{P_{m-1}(d_j(\xi'),\xi')}{\check{P}_{m-2,j}(d_j(\xi'),\xi')}t}}{i^{m-3}\,\check{P}_{m-2,j}(d_j(\xi'),\xi')}\,,
\end{align}
where $M$ is as in~\eqref{eq:moment} and $\xi'=\xi/\xii$. Then the asymptotic profile of the solution to~\eqref{eq:CPlin} is described by $I_{m-2}v$ if~$Q=Q_1$ and by $I_{m-3}w$ if $Q=Q_2$, where $I_a$ denotes the Riesz potential of order $a$, i.e., $\mathfrak{F}(I_af)=\xii^{-a}\hat f$ in appropriate sense. The expressions of $\hat v$ and $\hat w$ appear complicated but they become much easier in several cases of interest, see Examples~\ref{ex:MGT}, \ref{ex:BC}, \ref{ex:electric}.
\begin{theorem}\label{thm:asymp}
Assume that~$P_m(\partial_t,\partial_x)$ and~$P_{m-1}(\partial_t,\partial_x)$ are strictly hyperbolic, that the polynomials $P_{m-1}(\lambda,\xi')$ and $P_{m}(\lambda,\xi')$ strictly interlace for any~$\xi'\in S^{n-1}$ and that
\[ u_j\in H^{s_0-j} \cap L^1 \qquad \text{for some $s_0\geq m-1$, for $j=0,\ldots,m-1$.} \]
Then the solution $u$ to Cauchy problem~\eqref{eq:CPlin} with $Q=Q_1$ satisfies the following estimate:
\begin{equation}
\label{eq:Q1asymp}
\big\|\partial_t^k \big( u(t,\cdot) - I_{m-2} v(t,\cdot)\big)\big\|_{\dot H^s} = \textit{o}\big((1+t)^{-\frac{n}4-\frac{k+s-(m-2)}2}\big),
\end{equation}
for $k+s\leq s_0$, such that $n/2+k+s>m-2$, where $M$ is as in~\eqref{eq:moment} and $v$ is as in~\eqref{eq:v}. Assume, moreover, that~$P_{m-2}(\partial_t,\partial_x)$ is strictly hyperbolic and that the polynomials~$P_{m-2}(\lambda,\xi')$ and~$P_{m-1}(\lambda,\xi')$ strictly interlace for any~$\xi'\in S^{n-1}$. Then the solution to Cauchy problem~\eqref{eq:CPlin} with $Q=Q_2$ satisfies the following estimate:
\begin{equation}
\label{eq:Q2asymp}
\big\|\partial_t^k \big( u(t,\cdot) - I_{m-3}w(t,\cdot)\big)\big\|_{\dot H^s} = \textit{o}\big((1+t)^{-\frac{n}4-\frac{k+s-(m-3)}2}\big),
\end{equation}
for $k+s\leq s_0$, such that $n/2+k+s>m-3$, where $M$ is as in~\eqref{eq:moment} and $w$ is as in~\eqref{eq:w}.
\end{theorem}
When $M\neq0$, Theorem~\ref{thm:asymp} describes the asymptotic profile of the solution to~\eqref{eq:CPlin}. In particular, it guarantees the optimality of the decay rate~\eqref{eq:estQ1} in Theorem~\ref{thm:CPQ2} when $q=1$ and $M\neq0$, due to
\begin{align*}
\| I_{m-2}v(t,\cdot)\|_{\dot H^s}
    & = \Big\| \xii^{s-(m-2)}\,\sum_{j=1}^{m-1}\frac{e^{\xii^2\,\frac{P_m(b_j(\xi'),\xi')}{\check{P}_{m-1,j}(b_j(\xi'),\xi')}t}}{\check{P}_{m-1,j}(b_j(\xi'),\xi')} \Big\|_{L^2}\\
    & = t^{-\frac{n}4-\frac{s-(m-2)}2}\,\Big\|\xii^{s-(m-2)}\,\sum_{j=1}^{m-1} \frac{e^{\xii^2\frac{P_m(b_j(\xi'),\xi')}{\check{P}_{m-1,j}(b_j(\xi'),\xi')}}}{\check{P}_{m-1,j}(b_j(\xi'),\xi')} \Big\|_{L^2},
\end{align*}
when~$n/2+s>m-2$. We proceed similarly for the decay rate in~\eqref{eq:estQ2}, using $w$. For $\partial_t^k v$ and $\partial_t^kw$, we may proceed in a similar way, but in this case we shall carefully treat the cases in which $b_j(\xi')=0$ for some $\xi'\in S^{n-1}$, due to the different behavior of $|\lambda_j^k|$.

On the other hand, when $M=0$, Theorem~\ref{thm:asymp} implies that the estimates in Theorem~\ref{thm:CPQ2} are not optimal. Indeed, in this case, the estimates may be improved, see Remark~\ref{rem:M0}, and the asymptotic profile of the solution to~\eqref{eq:CPlin} may be described by assuming further moment conditions.
\begin{remark}
In the setting of linear and semilinear damped wave equations, the asymptotic profile of the solution, as $t\to\infty$, is described by the fundamental solution to a heat equation or, more in general, to a diffusive problem, a property called \emph{diffusion phenomenon}, see~\cite{MN03, Ni, N11as, HM00, W14sp}. The diffusion phenomenon also appears in wave and evolution models with structural damping~\cite{DAE14JDE, DAE17NA, K00}. In the case of higher order equations, the diffusion phenomenon does not generally appear, due the presence of the oscillating terms $e^{i\xii b_j(\xi')t}$ and $e^{i\xii d_j(\xi')t}$, in~\eqref{eq:v} and~\eqref{eq:w}. However, in the special case $m=3$ and $P_1=\lambda$, Theorem~\ref{thm:asymp} for $Q=Q_2$ may be improved to show that
\[ \|\partial_t^k (u(t,\cdot)-w(t,\cdot))\|_{\dot H^s} = \textit{o}\big((1+t)^{-\frac{n}4-\frac{s}2-k}\big), \]
where $w$ is $M$ times the fundamental solution to
\[ \partial_t w - \sum_{|\alpha|=2}c_{0,\alpha}\partial_x^\alpha w=0. \]
In particular, each time derivative produces an additional decay~$t^{-1}$ in place of $t^{-\frac12}$.

We stress that $\sum_{|\alpha|=2}c_{0,\alpha}\xi^\alpha>0$ for any $\xi\neq0$, as a consequence of the strict hyperbolicity of $P_2(\partial_t,\partial_x)$.
\end{remark}


\section{Some examples of physical interest}\label{sec:examples}

In this section, we provide a few examples of models with physical interest to which Theorems~\ref{thm:CPQ2} and~\ref{thm:asymp} and, more in general, the results in \textsection~\ref{sec:asymp}, directly apply. However, it is clear that our results apply to any hyperbolic homogeneous equation of order $m\geq3$, to which a hyperbolic homogeneous operator of order $m-1$ and, possibly, a hyperbolic homogeneous operator of order $m-2$, are added to produce a dissipation. In \textsection\ref{sec:weak}, two other examples are provided to discuss the case of weak interlacing when the lowest order term has a double root.
\begin{example}\label{ex:MGT}
The Jordan-Moore-Gibson-Thompson equation (see~\cite{J,KLM,K}) is a hyperbolic model for acoustic waves introduced to preserve the finite speed of propagation. Its linear version, called Moore-Gibson-Thompson (MGT) equation, is $L(\partial_t,\partial_x)u=f$, with
\[ L(\partial_t,\partial_x) = \tau \partial_t^3 + \partial_t^2 - c^2\Delta -(\tau c^2+b)\Delta \partial_t, \]
where $\tau>0$ is the intrinsic relaxation time of the heat flux~\cite{CFO}, $c>0$ is the sound speed, and $b>0$ is the diffusivity of sound. As $\tau\to0$, the equation reduces to the linear Kuznetsov~\cite{Kunetsov} or Westervelt~\cite{Westervelt} equation, i.e., a wave equation with viscoelastic damping~\cite{Ponce, S00}:
\begin{equation}\label{eq:visco}
\partial_t^2u - c^2\Delta u -b\Delta \partial_tu=0.
\end{equation}
Setting $L(\partial_t,\partial_x)u=\tau Q_1(\partial_t,\partial_x)u$, this corresponds to fix
\begin{align*}
P_3(\lambda,\xi')
    & = \lambda^3-(c^2+b\tau^{-1})\lambda=(\lambda+\sqrt{c^2+b\tau^{-1}}) \lambda (\lambda-\sqrt{c^2+b\tau^{-1}}),\\
P_2(\lambda,\xi')
    & = \tau^{-1}(\lambda^2-c^2) = \tau^{-1}\,(\lambda+c)(\lambda-c).
\end{align*}
In particular, $P_3$ and $P_2$ strictly interlace for any $b>0$, since the roots of $P_2(\lambda,\xi')$ are $(b_1,b_2)=(-c,c)$, whereas the roots of $P_3(\lambda,\xi')$ are
\[ a_1=- \sqrt{c^2+\frac{b}{\tau}}, \quad a_2=0,\quad a_3= \sqrt{c^2+\frac{b}{\tau}}. \]
%
We may now compute
\[ P_3(\pm c,\xi')=\mp c b\tau^{-1}, \qquad \check{P}_{2,1}(-c,\xi')=-2c\tau^{-1},\quad \check{P}_{2,2}(c,\xi')=2c\tau^{-1}, \]
so that, by Lemma~\ref{lem:low}, the $3$ roots $\lambda_j(\xi)$ of $Q(\lambda,i\xi)$ have the following asymptotic behavior as $\xi\to0$:
\[ \lambda_3(0)=-\tau^{-1}, \qquad \lambda_{1,2}(\xi) = \pm ic\xii - \frac{b}2\,\xii^2 + \textit{o}(\xii^2). \]
%
%
We stress that $\lambda_{1,2}$ have the same asymptotic behaviors, as~$\xi\to0$, of the roots of the viscoelastic damped wave equation obtained setting $\tau=0$ in the MGT equation:
\[ \lambda_\pm (\xi) = 
-\frac{b \xii^2 \pm i\xii \sqrt{4c^2-b^2\xii^2}}{2} = \pm ic\xii-\frac{b}2\,\xii^2+\textit{o}(\xii^2), \quad \xii\ll 1. 
\]
%
%
%
On the other hand, $\lambda_3(0)\to-\infty$ as~$\tau\to0$. 
%
%
We may write the asymptotic profile of the solution to the MGT equation as $I_1v$, where
\begin{align*}
\mathfrak{F}(I_1v)(t,\xi)
    & = \xii^{-1}\hat v(t,\xi) = M\,\frac{e^{ic\xii t-\frac{b}2\xii^2 t}-e^{-ic\xii t-\frac{b}2\xii^2 t}}{2ic\tau^{-1}\xii}= M\tau\,\frac{\sin(c\xii t)}{c\xii}\,e^{-\frac{b}2\xii^2 t},\\
M\tau
    & = \int_{\R^n} \big( \tau\,u_2(x) + u_1(x) \big) \,dx.
\end{align*}
By Theorem~\ref{thm:asymp}, we find
\[ \big\|\partial_t^k \big( u(t,\cdot) - I_1v(t,\cdot)\big)\big\|_{\dot H^s} = \textit{o}\big((1+t)^{-\frac{n}4-\frac{k+s-1}2}\big),\]
for $n/2+k+s>1$. As $\tau\to0$, we find the same kind of result obtained for the wave equation with viscoelastic damping~\eqref{eq:visco} in~\cite{I14}.
\end{example}
\begin{example}\label{ex:BC}
The Kuznetsov and Westervelt equations can formally be regarded as a simplification of the Blackstock-Crighton (BC) equation~\cite{Blackstock}. One of the model of the linear BC equation is
\begin{equation}\label{eq:BC}
\partial_t(\partial_t^2-c^2\Delta-b\Delta \partial_t)u-a\Delta(\partial_t^2-c^2\Delta)u  = 0,
\end{equation}
where $a>0$ is the thermal conductivity, whereas $c>0$ and $b>0$ are the sound speed and the diffusivity of sound, as in the Kuznetsov, Westervelt and MGT equations in Example~\ref{ex:BC}. Replacing the first term in the BC equation, that is, the time-derivative of the wave equation with viscoelastic damping, by the time-derivative of the MGT equation, we obtain the fourth order hyperbolic equation $\tau Q_1(\partial_t,\partial_x)u=0$, where
\begin{align*}
P_4(\lambda,\xi')
    & = \lambda^4-(c^2+(a+b)\tau^{-1})\lambda^2+ac^2\tau^{-1}=(\lambda+a_4)(\lambda+a_3)(\lambda-a_3)(\lambda-a_4), \\
P_3(\lambda,\xi')
    & = \tau^{-1}\lambda(\lambda^2-c^2) = \tau^{-1}\,(\lambda+c)\lambda(\lambda-c),
\end{align*}
where the roots of $P_4(\lambda,\xi')$ are
\[ a_{3,4}=\sqrt{\frac{c^2+(a+b)\tau^{-1}\pm \sqrt{(c^2+(a+b)\tau^{-1})^2-4a\tau^{-1}c^2}}2}, \quad a_2=-a_3, \quad a_1=-a_4,\]
and the roots of $P_3(\lambda,\xi')$ are $(b_1,b_2,b_3)=(-c,0,c)$, for any $\xi'\in S^{n-1}$. The strict interlacing condition is verified if, and only if, $b>0$. 
%
%
%
%
Due to
\begin{align*}
& P_4(\pm c,\xi')=bc^2\tau^{-1}, \quad P_4(0,\xi')=ac^2\tau^{-1},\\
& \check{P}_{3,1}(-c,\xi')=\check{P}_{3,3}(c,\xi')=2c^2\tau^{-1}, \quad \check{P}_{3,2}(0,\xi')=-c^2\tau^{-1},
\end{align*}
by Lemma~\ref{lem:low}, the four roots $\lambda_j(\xi)$ of $Q(\lambda,i\xi)$ have the following asymptotic behavior, as $\xi\to0$:
\begin{align*}
\lambda_4(0)
    & = -\tau^{-1},\\
\lambda_{1,3}(\xi)
    & = \pm ic\xii - \frac{b}2\,\xii^2 + \textit{o}(\xii^2), \\
\lambda_2(\xi)
    & = -a\,\xii^2 + \textit{o}(\xii^2).
\end{align*}
In particular, exception given for $\lambda_2$, the asymptotic behavior is the same of the roots of the MGT equation, see Example~\ref{ex:MGT}.

Setting~$\tau=0$, we may compare those asymptotic behaviors with the asymptotic behaviors of the roots of the BC equation~\eqref{eq:BC}. %
Recalling (see for instance, Routh-Hurwitz theorem) that a polynomial of third order with positive coefficients
\[ q(z)=z^3+\alpha_2z^2+\alpha_1z+\alpha_0 \]
is strictly stable if, and only if, $\alpha_0 < \alpha_1\alpha_2$, %
%
we find that~$Q(\lambda,i\xi)$, with
\[ Q(\lambda,i\xi) = \lambda(\lambda^2+c^2\xii^2+b\xii^2\lambda)+a\xii^2(\lambda^2+c^2\xii^2), \]
as in~\eqref{eq:BC}, is strictly stable for any~$\xi\neq0$ if, and only if, $b>0$. Following as in the proof of Lemma~\ref{lem:low}, it is easy to show that the three roots~$\lambda_0(\xi)$ and~$\lambda_\pm(\xi)$ of $Q(\lambda,i\xi)$ verify
\[ \lambda_\pm(\xi)= \pm ic\xii - \frac{b}2\,\xii^2 + \textit{o}(\xii^2) = \lambda_{1,3}(\xi),\qquad \lambda_0(\xi)=-a\,\xii^2 + \textit{o}(\xii^2)=\lambda_2(\xi), \]
as~$\xi\to0$. On the other hand, $\lambda_4(0)\to-\infty$ as~$\tau\to0$.

%

We may write the asymptotic profile of the solution to the BC equation as $I_2v$, where
\begin{align*}
\mathfrak{F}(I_2v)(t,\xi)
    & = \xii^{-2}\hat v(t,\xi) = M\,\frac{e^{ic\xii t-\frac{b}2\xii^2 t}+e^{-ic\xii t-\frac{b}2\xii^2 t}}{-2c^2\tau^{-1}\xii^2} + M\,\frac{e^{-a\xii^2 t}}{c^2\tau^{-1}\xii^2} \\
    & = \frac{M\tau}{c^2\xii^2}\left(e^{-a\xii^2 t}-\cos(c\xii t)\,e^{-\frac{b}2\xii^2 t}\right)\,,\\
M\tau
    & = \int_{\R^n} \big( \tau\,u_3(x) + u_2(x) \big) \,dx.
\end{align*}
By Theorem~\ref{thm:asymp}, we get
\[ \big\|\partial_t^k \big( u(t,\cdot) - I_2v(t,\cdot)\big)\big\|_{\dot H^s} = \textit{o}\big((1+t)^{-\frac{n}4-\frac{k+s-2}2}\big),\]
for $n/2+k+s>2$.
%
%
\end{example}

\begin{example}\label{ex:electric}
Let us consider the coupled system of elastic waves with Maxwell equations in~$\R^3$ (see, for instance, \cite{DLM09}), i.e.
\begin{equation}\label{eq:MHD3dplus}
\begin{cases}
\partial_t^2 u - \mu \Delta u - (\mu+\nu) \nabla \div u + \gamma \curl E =0,\\
\partial_t E +\sigma E - \curl H -\gamma\curl \partial_t u=0,\\
\partial_t H + c^2\curl E =0,
\end{cases}
\end{equation}
where~$\curl = \nabla\times$ denotes the curl operator, $\mu,\nu$ are the Lam\`e constants and verify $\mu>0$, $\mu+\nu>0$ and, after normalization, $\gamma>0$ is the coupling constant, $\sigma>0$ is the electric conductivity (we replaced $J=\sigma E$, by Ohm's law, in Maxwell's equations) and $c^2>0$ is the inverse of the product of the electric permittivity and magnetic permeability. Recalling that $\div \curl =0$, we notice that~$\div E$ verifies the scalar equation
\begin{equation}\label{eq:divE}
\partial_t \div E + \sigma \div E =0.
\end{equation}
Deriving with respect to $t$ the second equation in~\eqref{eq:MHD3dplus} and replacing the third one in it, recalling that
\begin{equation}\label{eq:curlcurl}
\curl \curl f = - \Delta f + \nabla \div f,
\end{equation}
we get the second-order $6\times 6$ system
\begin{equation}\label{eq:MHD3hom}
\begin{cases}
\partial_t^2 u - \mu \Delta u - (\mu+\nu) \nabla \div u + \gamma \curl E =0,\\
\partial_t^2 E +\sigma \partial_t E -c^2\Delta E + c^2\nabla \div E - \gamma \curl \partial_t^2 u=0.
\end{cases}
\end{equation}
Applying the operator~$ \partial_t^2-\mu\Delta-(\mu+\nu)\nabla \div $ to the second equation in~\eqref{eq:MHD3hom}, recalling~\eqref{eq:curlcurl}, we get a fourth-order $3\times3$ system for $E$:
\begin{equation}\label{eq:Evector}
(\partial_t^2 - \mu \Delta - (\mu+\nu) \nabla \div)(\partial_t^2+\sigma\partial_t-c^2\Delta+c^2\nabla \div)E + \gamma^2\partial_t^2 (-\Delta+\nabla\div) E=0.
\end{equation}
Since $\div E$ verifies~\eqref{eq:divE}, applying the operator~$\partial_t+\sigma$ to~\eqref{eq:Evector}, we may remove every term where $\div E$ appears and obtain the fifth-order scalar equation given by
\[ (\partial_t+\sigma) \big( (\partial_t^2-\mu\Delta)(\partial_t^2+\sigma\partial_t-c^2\Delta)-\gamma^2\partial_t^2\Delta\big) E=0, \]
which every component of $E$ shall satisfy. The equation above is $Q_2(\partial_t,\partial_x)E=0$, where $Q_2=P_5+P_4+P_3$ and
\begin{align*}
P_5(\lambda,\xi)
    & = \lambda\,\big(\lambda^4-(\mu+c^2+\gamma^2)\lambda^2\xii^2+c^2\mu\xii^4\big), \\
P_4(\lambda,\xi)
    & = \sigma\,\big(2\lambda^4-(2\mu+c^2+\gamma^2)\lambda^2\xii^2+c^2\mu\xii^4\big), \\
P_3(\lambda,\xi)
    & = \sigma^2 \lambda (\lambda^2-\mu\xii^2).
\end{align*}
The roots of $P_3(\lambda,\xi')$ are $(d_1,d_2,d_3)=(-\sqrt{\mu},0,\sqrt{\mu})$, the roots of $P_4(\lambda,\xi')$ are
\[ b_{3,4} = \frac12\,\sqrt{2\mu+c^2+\gamma^2\pm \sqrt{(2\mu+c^2+\gamma^2)^2-8c^2\mu}}, \qquad b_{1,2}=-b_{3,4}, \]
and the roots of $P_5(\lambda,\xi')$ are
\[ a_{4,5} = \frac1{\sqrt{2}}\,\sqrt{\mu+c^2+\gamma^2\pm \sqrt{(\mu+c^2+\gamma^2)^2-4c^2\mu}}, \qquad a_3=0, \quad a_{1,2}=-a_{4,5}. \]
The strict interlacing condition of $P_3(\lambda,\xi')$ and $P_4(\lambda,\xi')$ is verified if $b_3<1<b_4$, that is,
\begin{align*}
2\mu+c^2+\gamma^2-\sqrt{(2\mu+c^2+\gamma^2)^2-8c^2\mu} < 4\mu < 2\mu+c^2+\gamma^2+\sqrt{(2\mu+c^2+\gamma^2)^2-8c^2\mu},
\end{align*}
which holds if, and only if,
\[ \big(2\mu-(c^2+\gamma^2)\big)^2 < (2\mu+c^2+\gamma^2)^2-8c^2\mu, \]
i.e., $\gamma^2>0$. With long but straightforward calculations, one may also check that $P_4(\lambda,\xi')$ and $P_5(\lambda,\xi')$ strictly interlace. 
Therefore, Theorems~\ref{thm:CPQ2} and~\ref{thm:asymp} apply. Due to
\begin{align*}
& P_4(\pm\sqrt{\mu},\xi')=-\sigma\mu\gamma^2, \qquad P_4(0,\xi')=\sigma c^2\mu, \\
& \check{P}_{3,3}(\sqrt{\mu},\xi')=2\sigma^2\mu, \quad \check{P}_{3,2}(0,\xi')=-\sigma^2\mu, \quad \check{P}_{3,1}(-\sqrt{\mu},\xi')=2\sigma^2\mu,
\end{align*}
we may compute
\[ \lambda_{4,5}(0)=-\sigma, \qquad \lambda_{1,3}(\xi)=\pm i\xii\sqrt{\mu}-\frac{\gamma^2}{2\sigma}\xii^2+\textit{o}(\xii^2), \qquad \lambda_2(\xi)=-\frac{c^2}\sigma\,\xii^2+\textit{o}(\xii^2), \]
as $\xi\to0$. Therefore, the asymptotic profile of $E(t,\cdot)$, for each of its components, is described by $I_2w$, where
\begin{align*}
\mathfrak{F}(I_2w)(t,\xi)
    & = \xii^{-2}\hat w(t,\xi) = -\,M\,e^{-\frac{\gamma^2}{2\sigma}\xii^2t}\,\frac{e^{i\xii\sqrt{\mu}\,t}+e^{-i\xii\sqrt{\mu}\,t}}{2\mu\sigma^2\xii^2} + M\,\frac{e^{-\frac{c^2}\sigma\,\xii^2t}}{\mu\sigma^2\xii^2}\\
    & = \frac{M}{\mu\sigma^2\xii^2}\,\left(e^{-\frac{c^2}\sigma\,\xii^2t} - \cos(\xii\sqrt{\mu}\,t)\,e^{-\frac{\gamma^2}{2\sigma}\xii^2t} \right),
\intertext{if}
M
    & = \int_{\R^3} \big(\partial_t^4 E(0,x)+\sigma \partial_t^3 E(0,x)+\sigma^2 \partial_t^2 E(0,x)\big)\,dx \neq0.
\end{align*}
Explicitly,
\[ \|\partial_t^k\big( E(t,\cdot)-I_2w(t,\cdot)\big)\|_{\dot H^s} =\textit{o}\big((1+t)^{-\frac34-\frac{k+s-2}2}\big),\]
for any $k+s>1/2$.
\end{example}



\section{Proofs of the asymptotic profiles and of Theorems~\ref{thm:CPQ2} and~\ref{thm:asymp}}\label{sec:proofs}

We will first prove the asymptotic behavior of the roots of $Q(\lambda,i\xi)$ as $\xi\to0$ and as $\xii\to\infty$.

\subsection*{Proof of Lemmas~\ref{lem:low} and~\ref{lem:high}}

\begin{proof}[Proof of Lemma~\ref{lem:low}]
First of all, we notice that
\[ Q(\lambda,0)=\lambda^{m-\ell}\sum_{j=0}^{\ell} c_{m-j,0}\lambda^{\ell-j}, \]
so that~$\ell$ roots of $Q(\lambda,i\xi)$ solve~\eqref{eq:lambda0} at~$\xi=0$, whereas~$\lambda_j(0)=0$, for $j=1,\ldots,m-\ell$. We fix~$\xi'\in S^{n-1}$, and we set~$\xi=\rho\xi'$, with~$\rho>0$ and $\mu_j(\rho,\xi')=\rho^{-1}\lambda_j(\rho\xi')$. Since~$P_{m-\ell}$ is homogeneous,
\[ P_{m-\ell}(\lambda_j,i\xi) = \rho^{m-\ell} \, P_{m-\ell}(\mu_j,i\xi') = \rho^{m-\ell}\,c_{m-\ell,0}\,\prod_{k=1}^{m-\ell}(\mu_j-i\beta_k(\xi')). \]
Then we may write
\begin{equation}\label{eq:mainlow}
0 = \xii^{-(m-\ell)}\,Q(\lambda_j,i\xi)= P_{m-\ell}(\mu_j,i\xi') + \rho\,P_{m-\ell+1}(\mu_j,i\xi')+\textit{o}(\rho), \quad j=1,\ldots,m-\ell.
\end{equation}
In particular, as~$\rho\to0$, we obtain~$\mu_j = i\beta_j + \textit{o}(1)$, that is,
\[ \lambda_j(\xi)= i\xii \beta_j(\xi') + \textit{o}(\xii),\qquad j=1,\ldots,m-\ell.\]
On the other hand, since the roots $\beta_j(\xi')$ are distinct, due to the strict hyperbolicity of $P_{m-\ell}(\partial_t,\partial_x)$, the quantity
\[ \check{P}_{m-\ell,j}(\mu_j,i\xi') = c_{m-\ell,0} \prod_{k\neq j} (\mu_j-i\beta_k(\xi')), \]
is nonzero for sufficiently small $\xii$. Dividing~\eqref{eq:mainlow} by $\check P_{m-\ell,j}(\mu_j(\xi'),i\xi')$, we obtain
\[ \mu_j - i\beta_j(\xi') = -\rho\,\frac{P_{m-\ell+1}(\mu_j,i\xi')}{\check P_{m-\ell,j}(\mu_j,i\xi')} + \textit{o}(\rho) = \rho\,\frac{P_{m-\ell+1}(\beta_j(\xi'),\xi')}{\check P_{m-\ell,j}(\beta_j(\xi'),\xi')} + \textit{o}(\rho). \]
In the last equality we used that $P_{m-\ell+1}$ is homogeneous of degree $m-\ell+1$ and $\check{P}_{m-\ell,j}$ is homogeneous of degree $m-\ell-1$ to remove the imaginary unit $i$, with a sign change.

%
Multiplying by $\xii$ the above asymptotic behavior, we conclude the proof.
\end{proof}
\begin{proof}[Proof of Lemma~\ref{lem:high}]
We fix~$\xi'\in S^{n-1}$, and we set~$\xi=\rho\xi'$, with~$\rho>0$. We also define $\mu_j(\rho,\xi')=\rho^{-1}\lambda_j(\xi)$ for $j=1,\ldots,m$. By the homogeneity of~$P_m(\lambda,\xi)$,
\[ P_m(\lambda_j,i\xi) = \rho^m \, P_m(\mu_j,i\xi') = \rho^m\,\prod_{k=1}^m(\mu_j-ia_k(\xi')). \]
Then we may write
\begin{equation}\label{eq:mainhigh}
0 = \xii^{-m}\,Q_2(\lambda_j,i\xi)= P_m(\mu_j,i\xi') + \rho^{-1}\,P_{m-1}(\mu_j,i\xi') + \textit{o}(\rho^{-1}).
\end{equation}
%
%
In particular, as~$\rho\to\infty$, we obtain~$\mu_j = ia_j(\xi') + \textit{o}(1)$, that is, $\lambda_j(\xi)= i\xii a_j(\xi') + \textit{o}(\xii)$. On the other hand, since the roots $a_j(\xi')$ are distinct, due to the strict hyperbolicity of $P_{m}$, the quantity
\[ \check{P}_{m,j}(\mu_j,i\xi') = \prod_{k\neq j} (\mu_j-ia_k(\xi')), \]
is nonzero for sufficiently large $\xii$. Dividing~\eqref{eq:mainhigh} by $\check P_{m,j}(\mu_j,i\xi')$, we obtain
\[ \mu_j - ia_j = -\rho^{-1}\,\frac{P_{m-1}(\mu_j,i\xi')}{\check P_{m,j}(\mu_j,i\xi')} + \textit{o}(\rho^{-1}) = -\rho^{-1}\,\frac{P_{m-1}(a_j,\xi')}{\check P_{m,j}(a_j(\xi'),\xi')} + \textit{o}(\rho^{-1}). \]
In the last equality we used that $P_{m-1}$ and $\check{P}_{m,j}$ are both homogeneous of degree $m-1$ to remove the imaginary unit $i$.

Multiplying by $\xii$ the above asymptotic behavior, we conclude the proof.
\end{proof}


\subsection*{Proof of Theorems~\ref{thm:CPQ2} and~\ref{thm:asymp}}


After performing the Fourier transform with respect to the space variable~$x$ in~\eqref{eq:CPlin}, $\hat u(t,\cdot)=\mathfrak{F}u(t,\cdot)$ is the solution to the ODE problem
\begin{equation}\label{eq:CPF}
\begin{cases}
Q(\partial_t,i\xi) \hat u =0, \quad t\geq 0,\ \xi\in\R^n,\\
\partial_t^j\hat u(0,x)=\hat u_j(\xi),\quad j=0,\ldots,m-1,
\end{cases}
\end{equation}
We may write an explicit representation of the solution~$\hat u$ to~\eqref{eq:CPF}, in terms of the complex-valued roots~$\lambda_1(\xi)$, \ldots, $\lambda_m(\xi)$, of the polynomial~$Q(\lambda,i\xi)$. We first consider $\xi\neq0$ such that the roots are all distinct. In this case,
\begin{equation}\label{eq:urepgen}
\hat u(t,\xi) = \sum_{j=1}^m \frac{\hat u_{m-1}-\hat u_{m-2}\sum_{k\neq j}\lambda_k 
+ \ldots + (-1)^{m-1} \hat u_0 \prod_{k\neq j}\lambda_k}{\prod_{k\neq j}(\lambda_j-\lambda_k)}\,e^{\lambda_jt}.
\end{equation}
%
%
%
It is clear that it is impossible to compute, in general, the roots of a polynomial of order $m$ with complex-valued coefficients. For our purpose, it is sufficient to now that the real parts of the roots are negative, as a consequence of Lemmas~\ref{lem:stable1} and~\ref{lem:stable2}, together with the knowledge of the asymptotic profile of the roots as~$\xi\to0$ (Lemma~\ref{lem:low}) and~$\xii\to\infty$ (Lemma~\ref{lem:high}). This will provide us with pointwise estimates for
\begin{equation}\label{eq:urepgenest}
|\partial_t^k\hat u(t,\xi)| \leq \sum_{j=1}^m \frac{|\hat u_{m-1}|+|\hat u_{m-2}|\sum_{k\neq j}|\lambda_k| 
+ \ldots + |\hat u_0| \prod_{k\neq j}|\lambda_k|}{\prod_{k\neq j}|\lambda_k-\lambda_j|}\,|\lambda_j|^k\,e^{\Re\lambda_jt},
\end{equation}
for any $\xi\neq0$.
\begin{remark}\label{rem:multiple}
If there is some multiple root, representation~\eqref{eq:urepgen} and estimate~\eqref{eq:urepgenest} are conveniently modified. 
For instance, let $m=3$ and assume that $\lambda_1(\xi)\neq\lambda_2(\xi)=\lambda_3(\xi)$, for some~$\xi\neq0$. Then~\eqref{eq:urepgen} is replaced by
\begin{equation}\label{eq:urepdouble}\begin{split}
\hat u(t,\xi)
    & = \frac{\hat u_2 - 2\lambda_3 \hat u_1 + \lambda_3^2\hat u_0}{(\lambda_3-\lambda_1)^2}\,e^{\lambda_1t} \\
    & \qquad - \frac{\hat u_2-2\lambda_3 \hat u_1 +\lambda_1(2\lambda_3-\lambda_1)\hat u_0}{(\lambda_1-\lambda_3)^2}\,e^{\lambda_3t} \\
    & \qquad - \frac{\hat u_2 - (\lambda_1+\lambda_3) \hat u_1+\lambda_1\lambda_3\hat u_0}{\lambda_1-\lambda_3}\, t\, e^{\lambda_3t}.
\end{split}\end{equation}
Formula~\eqref{eq:urepdouble} is the limit of~\eqref{eq:urepgen} as $\lambda_2\to\lambda_3$. Indeed, writing formula~\eqref{eq:urepgen} as
\begin{equation}\label{eq:urepdoublenear}\begin{split}
\hat u(t,\xi)
    & = \frac{\hat u_2 - (\lambda_2+\lambda_3) \hat u_1 + \lambda_2\lambda_3\hat u_0}{(\lambda_2-\lambda_1)(\lambda_3-\lambda_1)}\,e^{\lambda_1t}\\
    & \quad + \left(\frac{\hat u_2 - (\lambda_1+\lambda_3) \hat u_1 + \lambda_1\lambda_3\hat u_0}{(\lambda_1-\lambda_2)(\lambda_3-\lambda_2)}+\frac{\hat u_2 - (\lambda_1+\lambda_2) \hat u_1 + \lambda_1\lambda_2\hat u_0}{(\lambda_1-\lambda_3)(\lambda_2-\lambda_3)}\right)\,e^{\lambda_2t}\\
    & \quad + \frac{\hat u_2 - (\lambda_1+\lambda_2) \hat u_1 + \lambda_1\lambda_2\hat u_0}{\lambda_1-\lambda_3}\,\frac{e^{\lambda_3t}-e^{\lambda_2t}}{\lambda_2-\lambda_3},
\end{split}\end{equation}
rewriting the second term as
\[ - \frac{\hat u_2 - (\lambda_2+\lambda_3) \hat u_1 + \lambda_1(\lambda_3+\lambda_2-\lambda_1)\hat u_0}{(\lambda_1-\lambda_2)(\lambda_1-\lambda_3)}\,e^{\lambda_2t}, \]
and taking the limit as $\lambda_2\to\lambda_3$, we find~\eqref{eq:urepdouble}. To treat such a case, estimate~\eqref{eq:urepgenest} may be replaced by
\begin{align*}
|\hat u(t,\xi)|
    & \leq \frac{|\hat u_2|+|\lambda_2+\lambda_3|\,|\hat u_1| + |\lambda_2\lambda_3| |\hat u_0|}{|\lambda_2-\lambda_1|\,|\lambda_3-\lambda_1|}\,e^{\Re\lambda_1t}\\
    & \qquad + \frac{|\hat u_2|+|\lambda_2+\lambda_3|\,|\hat u_1| + |\lambda_1|\,|\lambda_3+\lambda_2-\lambda_1|\,|\hat u_0|}{|\lambda_1-\lambda_2|\,|\lambda_1-\lambda_3|}\,e^{\Re\lambda_2t}\\
    & \qquad + \frac{|\hat u_2|+|\lambda_1+\lambda_2|\,|\hat u_1| + |\lambda_1\lambda_2|\,|\hat u_0|}{|\lambda_1-\lambda_3|}\,t\,e^{\Re\lambda_3t}\,,
\end{align*}
and similarly for the time derivatives. For the sake of brevity, we omit a general formula for $m\geq3$ and multiple roots.
\end{remark}
We are now ready to prove Theorem~\ref{thm:CPQ2}.
\begin{proof}[Proof of Theorem~\ref{thm:CPQ2}]
The Cauchy problem~\eqref{eq:CPlin} is well-posed in $H^s$ since $P_m(\partial_t,\partial_x)$ is strictly hyperbolic, so we shall only prove the desired decay rate in~\eqref{eq:estQ1} and in~\eqref{eq:estQ2} as $t\to\infty$. Therefore, we assume $t\geq1$ in the following. We use Plancherel's theorem, Riemann-Lebesgue theorem and Hausdorff-Young inequality to estimate the $\dot H^s$ norm of $\partial_t^k u(t,\cdot)$, using pointwise estimates for $|\partial_t^k\hat u(t,\xi)|$.

Exception given for the case $Q=Q_2$ and $4c_{m-2,0}=c_{m-1,0}^2$, we may fix sufficiently small~$\delta>0$ and~$c>0$, such that the roots~$\lambda_j(\xi)$, described in Lemmas~\ref{lem:low} are distinct for~$0<\xii\leq\delta$, and, in view of \eqref{eq:strictlow}, for any~$\xii\leq\delta$, it holds:
\[ \Re\lambda_m(\xi)\leq -c, \qquad \Re\lambda_j(\xi)\leq -c\xii^2,\qquad j=1,\ldots,m-1, \]
if $Q=Q_1$, or
\[ \Re\lambda_{m-1,m}(\xi)\leq -c,\qquad \Re\lambda_j(\xi)\leq -c\xii^2, \qquad j=1,\ldots,m-2, \]
if $Q=Q_2$, for some $c>0$.

We first consider $Q=Q_1$. According to Lemma~\ref{lem:low}, we may estimate
\[ |\lambda_m(\xi)|\lesssim1, \qquad |\lambda_j(\xi)| \lesssim \xii, \qquad j=1,\ldots,m-1, \]
as well as
\begin{align*}
|\lambda_k(\xi)-\lambda_m(\xi)|
    & = c_{m-1,0}+\textit{o}(1) \gtrsim 1, \qquad k=1,\ldots,m-1,\\
|\lambda_k(\xi)-\lambda_j(\xi)|
    & =|b_k(\xi')-b_j(\xi')|\,\xii+\textit{o}(\xii)\gtrsim \xii,\qquad j,k=1,\ldots,m-1, \ k\neq j.
\end{align*}
As a consequence, we obtain
\begin{align*}
& \frac1{\prod_{k\neq m}|\lambda_k(\xi)-\lambda_m(\xi)|} \lesssim 1,\\
& \frac1{\prod_{k\neq j}|\lambda_k(\xi)-\lambda_j(\xi)|} \lesssim \xii^{-(m-2)}, \qquad j=1,\ldots,m-1,
\end{align*}
and
\begin{align*}
& \sum_{k\neq m}|\lambda_k(\xi)|=\sum_{k\neq m}|b_k(\xi)|+\textit{o}(\xii)\lesssim \xii, \qquad \ldots, \quad \prod_{k\neq m}|\lambda_k| \lesssim \xii^{m-1},\\
& \sum_{k\neq j}|\lambda_k|=c_{m-1,0}+\textit{o}(1)\lesssim 1, \qquad \ldots, \quad \prod_{k\neq j}|\lambda_k| \lesssim \xii^{m-2},\quad j=1,\ldots,m-1.
\end{align*}
%
%
%
Recalling~\eqref{eq:urepgen}, we get
\begin{align*}
|\partial_t^k\hat u(t,\xi)|
    & \lesssim \xii^{k-(m-2)}\,\big(|\hat u_{m-1}(\xi)|+|\hat u_{m-2}(\xi)|\big)\,e^{-c\xii^2t} + \sum_{j=0}^{m-3} \xii^{k-j}\,|\hat u_j(\xi)|\,e^{-c\xii^2t} \\
    & \qquad + \sum_{j=0}^{m-1} \xii^{m-1-j}\,|\hat u_j(\xi)|\,\,e^{-ct},
\end{align*}
for~$k\geq0$. 
%
%
If~$q=2$, we immediately obtain
\begin{align*}
\|\xii^s\partial_t^k \hat u(t,\cdot)\|_{L^2(\xii\leq\delta)}
    & \lesssim \big(\|u_{m-1}\|_{L^2}+\|u_{m-2}\|_{L^2}\big)\,\sup_{\xii\leq\delta} \xii^{k+s-(m-2)}\,e^{-c\xii^2t} \\
    & \qquad + \sum_{j=0}^{m-3} \|u_j\|_{L^2}\,\sup_{\xii\leq\delta} \xii^{k+s-j}\,e^{-c\xii^2t} + e^{-ct}\,\sum_{j=0}^{m-1} \|u_j\|_{L^2}\\
    & \lesssim t^{-\frac{k+s-(m-2)}2}\,\big(\|u_{m-1}\|_{L^2}+\|u_{m-2}\|_{L^2}\big)+ \sum_{j=0}^{m-3} t^{-\frac{k+s-j}2} \|u_j\|_{L^2},
\end{align*}
provided that $k+s\geq m-2$.

Let $q\in[1,2)$ and set $q'=q/(q-1)\in(2,\infty]$ and $r=1/2-1/q'=1/q-1/2$. By the change of variable~$\eta=\sqrt{t}\xi$ and by H\"older inequality, we easily get
\begin{align*}
\|\xii^s\partial_t^k\hat u(t,\cdot)\|_{L^2(\xii\leq\delta)}
    & \lesssim \Big( \int_{\xii\leq\delta} \xii^{r(k+s-(m-2))}\,e^{-rc\xii^2t}\,d\xi\Big)^{\frac1r} \,\big(\|\hat u_{m-1}\|_{L^{q'}}+\|\hat u_{m-2}\|_{L^{q'}}\big) \\
    & \qquad + \sum_{j=0}^{m-3} \Big( \int_{\xii\leq\delta} \xii^{r(k+s-j)}\,e^{-rc\xii^2t}\,d\xi\Big)^{\frac1r} \,\|\hat u_j\|_{L^{q'}} + e^{-ct}\,\sum_{j=0}^{m-1} \|u_j\|_{L^2} \\
    & \lesssim t^{-\frac{n}r-(k+s-(m-2))}\,\big(\|u_{m-1}\|_{L^q}+\|u_{m-2}\|_{L^q}\big)+\sum_{j=0}^{m-3}t^{-\frac{n}r-(k+s-j)}\,\|u_j\|_{L^q},
\end{align*}
provided that~$n/r+k+s>m-2$, so that the power~$\xii^{r(k+s-(m-2))}$ is integrable near~$\xi=0$.

We now consider $Q=Q_2$. Let~$j_0=m-1,m$ in the following, for brevity. According to Lemma~\ref{lem:low}, we may estimate
\[ |\lambda_{j_0}(\xi)|\lesssim1, \qquad |\lambda_j(\xi)| \lesssim \xii, \qquad j=1,\ldots,m-2, \]
as well as
\[ |\lambda_k(\xi)-\lambda_{j_0}(\xi)| \gtrsim 1, \qquad |\lambda_k(\xi)-\lambda_j(\xi)|\gtrsim \xii,\qquad j,k=1,\ldots,m-2, \ k\neq j. \]
Moreover, $|\lambda_m(\xi)-\lambda_{m-1}(\xi)| \gtrsim 1$, since we assumed $4c_{m-2,0}\neq c_{m-1,0}^2$. As a consequence, we obtain
\begin{align*}
& \frac1{\prod_{k\neq j_0}|\lambda_k-\lambda_{j_0}|} \lesssim 1, \\
& \frac1{\prod_{k\neq j}|\lambda_k-\lambda_j|} \lesssim \xii^{-(m-3)}, \qquad j=1,\ldots,m-2,
\end{align*}
and
\begin{align*}
& \sum_{k\neq j_0}|\lambda_k|\lesssim 1, \quad \sum_{k,\ell\neq j_0, k<\ell}|\lambda_k\lambda_\ell|\lesssim \xii, \quad \ldots,\quad \prod_{k\neq j_0}|\lambda_k| \lesssim \xii^{m-2},\\
& \sum_{k\neq j}|\lambda_k|\lesssim 1, \quad \sum_{k,\ell\neq j, k<\ell}|\lambda_k\lambda_\ell|\lesssim 1, \quad \ldots, \quad \prod_{k\neq j}|\lambda_k| \lesssim \xii^{m-3},\quad j=1,\ldots,m-2.
\end{align*}
Recalling~\eqref{eq:urepgen}, we get
\begin{align*}
|\partial_t^k\hat u(t,\xi)|
    & \lesssim \xii^{k-(m-3)}\,\big(|\hat u_{m-1}(\xi)|+|\hat u_{m-2}(\xi)|+|\hat u_{m-3}(\xi)|\big)\,e^{-c\xii^2t} + \sum_{j=0}^{m-4} \xii^{k-j}\,|\hat u_j(\xi)|\,e^{-c\xii^2t} \\
    & \qquad + |\hat u_{m-1}(\xi)|\,\,e^{-ct} + \sum_{j=0}^{m-2} \xii^{m-2-j}\,|\hat u_j(\xi)|\,\,e^{-ct},
\end{align*}
for~$k\geq0$. 
%
%
We proceed as we did for $Q_1$. If~$q=2$, we immediately obtain
\begin{align*}
\|\xii^s\partial_t^k \hat u(t,\cdot)\|_{L^2(\xii\leq\delta)}
    & \lesssim \big(\|u_{m-1}\|_{L^2}+\|u_{m-2}\|_{L^2}+\|u_{m-3}\|_{L^2}\big)\,\sup_{\xii\leq\delta} \xii^{k+s-(m-3)}\,e^{-c\xii^2t} \\
    & \qquad + \sum_{j=0}^{m-4} \|u_j\|_{L^2}\,\sup_{\xii\leq\delta} \xii^{k+s-j}\,e^{-c\xii^2t} + e^{-ct}\,\sum_{j=0}^{m-1} \|u_j\|_{L^2}\\
    & \lesssim t^{-\frac{k+s-(m-3)}2}\,\big(\|u_{m-1}\|_{L^2}+\|u_{m-2}\|_{L^2}+\|u_{m-3}\|_{L^2}\big)+ \sum_{j=0}^{m-4} t^{-\frac{k+s-j}2} \|u_j\|_{L^2},
\end{align*}
provided that $k+s\geq m-3$. Let $q\in[1,2)$ and set $q'=q/(q-1)\in(2,\infty]$ and $r=1/2-1/q'=1/q-1/2$. By the change of variable~$\eta=\sqrt{t}\xi$ and by H\"older inequality, we may easily get
\begin{align*}
\|\xii^s\partial_t^k\hat u(t,\cdot)\|_{L^2(\xii\leq\delta)}
    & \lesssim \Big( \int_{\xii\leq\delta} \xii^{r(k+s-(m-3))}\,e^{-rc\xii^2t}\,d\xi\Big)^{\frac1r} \,\big(\|\hat u_{m-1}\|_{L^{q'}}+\|\hat u_{m-2}\|_{L^{q'}}+\|\hat u_{m-3}\|_{L^{q'}}\big) \\
    & \qquad + \sum_{j=0}^{m-4} \Big( \int_{\xii\leq\delta} \xii^{r(k+s-j)}\,e^{-rc\xii^2t}\,d\xi\Big)^{\frac1r} \,\|\hat u_j\|_{L^{q'}} + e^{-ct}\,\sum_{j=0}^{m-1} \|u_j\|_{L^2} \\
    & \lesssim t^{-\frac{n}r-(k+s-(m-3))}\,\big(\|u_{m-1}\|_{L^q}+\|u_{m-2}\|_{L^q}+\|u_{m-3}\|_{L^q}\big)+\sum_{j=0}^{m-4}t^{-\frac{n}r-(k+s-j)}\,\|u_j\|_{L^q},
\end{align*}
provided that~$n/r+k+s>m-3$, so that the power~$\xii^{r(k+s-(m-3))}$ is integrable near~$\xi=0$.

\bigskip

We now consider the high frequencies. Since $P_m(\partial_t,\partial_x)$ is strictly hyperbolic, there exist~$M>0$ and~$c>0$ such that the roots~$\lambda_j$ in Lemma~\ref{lem:high} are distinct for any~$\xii\geq M$, and, for any~$\xii\geq M$, it holds~$\Re\lambda_j(\xi)\leq -c$ for some $c>0$. Moreover, we may estimate
\[ |\lambda_j(\xi)| \lesssim \xii,\qquad |\lambda_j(\xi)-\lambda_\ell(\xi)|=|a_j(\xi')-a_\ell(\xi')|\,\xii+\textit{o}(\xii) \gtrsim \xii,\quad j\neq \ell. \]
Recalling~\eqref{eq:urepgen}, we may now estimate
\[ |\partial_t^k\hat u(t,\xi)| \lesssim \sum_{j=0}^{m-1} \xii^{k-j} |\hat u_j(\xi)|\,e^{-ct}, \]
for~$k\geq0$. We immediately obtain that
\[ \Big(\int_{\xii\geq M} \xii^{2s} |\partial_t^k\hat u(t,\xi)|^2\,d\xi\Big)^{\frac12} \lesssim e^{-ct}\,\sum_{j=0}^{m-1} \|u_j\|_{H^{s+k-j}}. \]
At intermediate frequencies, $\delta\leq \xii\leq M$, we shall not worry about regularity of the initial data or about the decay rate, since~$\Re\lambda(\xi)\leq -c$, for some~$c>0$, uniformly in a compact subset of~$\R^n\setminus\{0\}$. The only difference with the previous analysis is that the roots~$\lambda_j$ may coincide for some~$\xi$. However, this possibility has no influence on the exponential decay rate. Indeed, following as in Remark~\ref{rem:multiple}, if at most $p$ roots coincide at some point $\xi\neq0$, we get the estimate
\[ |\partial_t^k \hat u(t,\xi)| \lesssim (1+t)^{p-1}\,e^{-ct}\sum_{j=0}^{m-1}|\hat u_j(\xi)|, \]
%
but it is clear that~$(1+t)^{m-1}\,e^{-ct} \leq C e^{-c't}$, for any fixed~$c'\in(0,c)$, for some $C>0$, so that the decay remains exponential even if multiple roots are present.



Summarizing, we may unify the estimate at high and intermediate frequencies, getting
\[ \Big(\int_{\xii\geq \delta} \xii^{2s} |\partial_t^k\hat u(t,\xi)|^2\,d\xi\Big)^{\frac12} \lesssim e^{-c't}\,\sum_{j=0}^{m-1} \|u_j\|_{H^{s+k-j}}. \]
and this concludes the proof of~\eqref{eq:estQ1} and~\eqref{eq:estQ2}.

We now go back to the case that we excluded, $4c_{m-2,0}=c_{m-1,0}^2$. In this case, we do not rely on~\eqref{eq:urepgen} for $\lambda_{m-1}$ and $\lambda_m$ near $\xi=0$, since $\lambda_m(0)=\lambda_{m-1}(0)=-c_{m-1,0}/2$. However, taking into account of Remark~\ref{rem:multiple}, the only difference is that $e^{-ct}$ is replaced by $(1+t)\,e^{-ct}$ in the estimate of $\hat u(t,\xi)$ related to these two roots, but as seen before, this is not a big deal, since $(1+t)\,e^{-ct} \leq C e^{-c't}$, for any fixed~$c'\in(0,c)$, for some $C>0$. This concludes the proof.
%
\end{proof}


With minor modifications in the first part of the proof of Theorem~\ref{thm:CPQ2}, we may prove Theorem~\ref{thm:asymp}.
\begin{proof}[Proof of Theorem~\ref{thm:asymp}]
We prove Theorem~\ref{thm:asymp} for $Q=Q_1$, being the proof for $Q=Q_2$ completely analogous.

In view of Theorem~\ref{thm:CPQ2}, we may assume with no loss of generality that $u_0=\ldots=u_{m-3}=0$. We claim that if
\[ u_{m-1},u_{m-2}\in L^{1,1}=L^1((1+|x|) dx), \]
then for sufficiently small $\delta>0$ and for any $t\geq1$, we have the improved estimate:
\begin{equation}\label{claim:11}
\big\|\xii^s\partial_t^k\big(\hat u(t,\xi)-\xii^{-(m-2)}\hat v(t,\xi)\big)\big\|_{L^2(\xii\leq\delta)} \lesssim t^{-\frac{n}4-(k+s-(m-2))-\frac12}\,\big(\|u_{m-1}\|_{L^{1,1}}+\|u_{m-2}\|_{L^{1,1}}\big),
\end{equation}
where an additional power $t^{-\frac12}$ appears, with respect to the decay rate in~\eqref{eq:estQ1}, provided that~$n/2+1+k+s>m-2$. By claim~\eqref{claim:11}, the proof of Theorem~\ref{thm:asymp} follows by a density argument, as in~\cite[Lemma 3.2]{K00}. Namely, for given $u_{m-1},u_{m-2}\in L^1$, for any $\varepsilon$, we choose $u_{m-1}^\varepsilon,u_{m-2}^\varepsilon\in L^{1,1}$ such that
\[ \|u_{m-1}-u_{m-1}^\varepsilon\|_{L^1}+\|u_{m-2}-u_{m-2}^\varepsilon\|_{L^1}\leq\varepsilon.\]
By combining Theorem~\ref{thm:CPQ2} with initial data $u_j-u_j^\varepsilon$ and claim~\eqref{claim:11}, for any sufficiently large $T_\varepsilon$, we derive
\begin{align*}
& t^{\frac{n}4+(k+s-(m-2))}\,\big\|\xii^s\partial_t^k\big(\hat u(t,\xi)-\xii^{-(m-2)}\hat v(t,\xi)\big)\big\|_{L^2(\xii\leq\delta)}\\
& \qquad \lesssim t^{-\frac12}\,\big(\|u_{m-1}^\varepsilon\|_{L^{1,1}}+\|u_{m-2}^\varepsilon\|_{L^{1,1}}\big) + \varepsilon\leq 2\varepsilon, \qquad t\geq T_\varepsilon;
\end{align*}
hence, the proof of Theorem~\ref{thm:asymp} follows. We now prove claim~\eqref{claim:11}. If $f\in L^{1,1}$, then $\hat f\in \mathcal C^1$; in particular, $|\hat f(\xi)-\hat f(0)|\lesssim\xii\,\|\hat f\|_{L^{1,1}}$, by Lagrange theorem. Due to
\[ M= \hat u_{m-1}(0) + c_{m-1,0}\hat u_{m-2}(0), \]
for any $j=1,\ldots,m-1$, we may estimate
\[ \big| \hat u_{m-1}(\xi) - \hat u_{m-2}(\xi)\sum_{k\neq j}\lambda_k(\xi) -M\big| \lesssim \xii\,\big(\|u_{m-1}\|_{L^{1,1}}+\|u_{m-2}\|_{L^{1,1}}\big), \]
as $\xi\to0$. Indeed, $|\lambda_k(\xi)|\lesssim \xii$ for $k\neq m$ and $\lambda_m(0)=-c_{m-1,0}$. On the other hand, for any $j=1,\ldots,m-1$, we may estimate
\begin{align*}
& \Big| e^{\lambda_jt}-e^{i\xii\,b_j(\xi')t+\xii^2\,\frac{P_m(b_j(\xi'),\xi')}{\check{P}_{m-1,j}(b_j(\xi'),\xi')}t}\Big| \lesssim t\xii^3\,e^{-c\xii^2t}, \\
& \left| \frac{1}{(i\xii)^{m-2}\check{P}_{m-1,j}(b_j(\xi'),\xi')}-\frac1{\prod_{k\neq j}(\lambda_j(\xi)-\lambda_k(\xi))}\right| \\
& \qquad = \left| \frac1{(i\xii)^{m-2}c_{m-1,0}\prod_{k\neq j}(b_j(\xi')-b_k(\xi'))}-\frac1{\prod_{k\neq j}(\lambda_j(\xi)-\lambda_k(\xi))}\right| \lesssim \xii^{1-(m-2)},
\end{align*}
where we replaced
\[ \check{P}_{m-1,j}(b_j(\xi'),\xi') = c_{m-1,0}\,\prod_{k\neq j}(b_j(\xi')-b_k(\xi')), \]
and we used Lemma~\ref{lem:low} to estimate
\begin{align*}
& |\lambda_j(\xi)-\lambda_m(\xi)-c_{m-1,0}|\lesssim \xii,\\
& |\lambda_j(\xi)-\lambda_k(\xi)-i\xii (b_j(\xi')-b_k(\xi'))|\lesssim \xii^2, \qquad k=1,\ldots,m-2, \ k\neq j.
\end{align*}
Therefore, for $\xii\leq\delta$, we get the improved estimate
\[ \big|\partial_t^k \big(\hat u(t,\xi)-\xii^{-(m-2)}\hat v(t,\xi)\big) \big|\lesssim \xii^{1+k-(m-2)}\,(1+t\xii^2)\,e^{-c\xii^2t}\,\big(\|u_{m-1}\|_{L^{1,1}}+\|u_{m-2}\|_{L^{1,1}}\big), \]
from which~\eqref{claim:11} follows, as in the proof of Theorem~\ref{thm:CPQ2}. We mention that when $Q=Q_2$,
\begin{align*}
& \left| \frac{1}{(i\xii)^{m-3}\check{P}_{m-2,j}(d_j(\xi'),\xi')}-\frac1{\prod_{k\neq j}(\lambda_j(\xi)-\lambda_k(\xi))}\right| \\
& \qquad = \left| \frac1{(i\xii)^{m-3}c_{m-2,0}\prod_{k\neq j}(d_j(\xi')-d_k(\xi'))}-\frac1{\prod_{k\neq j}(\lambda_j(\xi)-\lambda_k(\xi))}\right| \lesssim \xii^{1-(m-3)},
\end{align*}
for $j=1,\ldots,m-2$, where we estimated
\[ |(\lambda_j-\lambda_m)(\lambda_j-\lambda_{m-1})-c_{m-2,0}| \lesssim \xii. \]
\end{proof}
\begin{remark}\label{rem:M0}
The proof of Theorem~\ref{thm:asymp} shows that, when $M=0$, estimates~\eqref{eq:estQ1} and~\eqref{eq:estQ2} for $q=1$ may be improved, assuming some initial data in $L^{1,1}$. In particular, \eqref{eq:estQ1} is improved to
\begin{equation}
\label{eq:estQ1M0}\begin{split}
\|\partial_t^ku(t,\cdot)\|_{\dot H^s}
    & \leq C\,(1+t)^{-\frac{n}4-\frac{k+s-(m-3)}2}\,\sum_{j=m-2,m-1}\big(\|u_j\|_{L^{1,1}}+\|u_j\|_{H^{k+s-j}}\big) \\
    & \qquad + C\,\sum_{j=0}^{m-3}(1+t)^{-\frac{n}4-\frac{k+s-j}2}\,\big(\|u_j\|_{L^1}+\|u_j\|_{H^{k+s-j}}\big),
\end{split}\end{equation}
for any $k+s\leq s_0$, such that $n/2+k+s>m-3$, and \eqref{eq:estQ2} is improved to
\begin{equation}
\label{eq:estQ2M0}\begin{split}
\|\partial_t^k u(t,\cdot)\|_{\dot H^s}
    & \leq C\,(1+t)^{-\frac{n}4-\frac{k+s-(m-4)}2}\,\sum_{j=m-3}^{m-1}\big(\|u_j\|_{L^{1,1}}+\|u_j\|_{H^{k+s-j}}\big)\\
     & \qquad + C\,\sum_{j=0}^{m-4}(1+t)^{-\frac{n}4-\frac{k+s-j}2}\,\big(\|u_j\|_{L^1}+\|u_j\|_{H^{k+s-j}}\big),
\end{split}
\end{equation}
for any $k+s\leq s_0$, such that $n/2+k+s>m-4$.
\end{remark}
\begin{remark}\label{rem:genQ}
In this paper, the choice to consider one or two lower order homogeneous hyperbolic operators is due to the possibility to use Lemmas~\ref{lem:stable1} and~\ref{lem:stable2} to express a condition equivalent to the strict stability of $Q(\lambda,i\xi)$. If one already knows that a more general operator $Q$ as in~\eqref{eq:Q} verifies the property that $Q(\lambda,i\xi)$ is strictly stable (as in the next Example~\ref{ex:3}), Lemmas~\ref{lem:low} and Lemma~\ref{lem:high} may be applied to obtain a result analogous to Theorems~\ref{thm:CPQ2} and~\ref{thm:asymp} for the solution to~\eqref{eq:CPlin}. For instance, the strict stability of $Q(\lambda,i\xi)$ could be checked by using Routh-Hurwitz criteria if all coefficients of $Q(\lambda,i\xi)$ are real.
\end{remark}
\begin{example}\label{ex:3}
Let us define
\[ Q(\partial_t,\partial_x)=\sum_{j=0}^3 P_{m-j}(\partial_t,\partial_x). \]
By Hermite--Biehler theorem, $Q(\lambda,i\xi)$ is strictly stable if, and only if, $P_{m-1}(\lambda,\xi)-P_{m-3}(\lambda,\xi)$ and $P_m(\lambda,\xi)-P_{m-2}(\lambda,\xi)$ strictly interlace, for any~$\xi\neq0$. Letting $\xi\to0$, we find the necessary condition that $c_{m-3,0}\leq c_{m-2,0}\,c_{m-1,0}$, which correspond to the weak stability (i.e., the real part of the roots is nonpositive) of
\[ Q(\lambda,0)=\lambda^{m-3}(\lambda^3+c_{m-1,0}\lambda^2+c_{m-2,0}\lambda+c_{m-3,0}).\]
We stress that the strict stability of $Q(\lambda,i\xi)$ is no longer only dependent on the roots of the homogeneous polynomials $P_{m-j}(\lambda,\xi')$, but it also involves the multiplicative coefficients. This is due to the fact that the polynomial has more than two lower order terms.

We provide an easy example, for which we may directly check the strict stability of $Q(\lambda,i\xi)$. We fix:
\begin{align*}
P_4(\partial_t,\partial_x)
    & = \partial_t^2 (\partial_t^2-a^2\Delta),\\
P_3(\partial_t,\partial_x)
    & = c_{3}\,\partial_t (\partial_t^2-a^2\Delta),\\
P_2(\partial_t,\partial_x)
    & = c_{2}\,(\partial_t^2-b^2\Delta),\\
P_1(\partial_t,\partial_x)
    & = c_{1}\,\partial_t,
\end{align*}
where $c_1,c_2,c_3,a,b>0$. Due to
\begin{align*}
P_4(\lambda,\xi)-P_2(\lambda,\xi)
    & = \lambda^4 - (c_2+a^2\xii^2) \lambda^2 + c_2b^2\xii^2, \\
P_3(\lambda,\xi)-P_1(\lambda,\xi)
    & = c_3 \lambda^3 - (c_1+c_3a^2\xii^2)\lambda,
\end{align*}
by straightforward computation, $P_3(\lambda,\xi)-P_1(\lambda,\xi)$ and $P_4(\lambda,\xi)-P_2(\lambda,\xi)$ strictly interlace for any~$\xi\neq0$ if, and only if, %
%
%
$c_1<c_2c_3$ and
\[ b^2 < \left(1-\frac{c_1}{c_2c_3}\right)\,a^2. \]
Applying Lemma~\ref{lem:low}, we find that $\lambda_2(0)$, $\lambda_3(0)$, $\lambda_4(0)$ are the solutions to $\lambda^3+c_3\lambda^2+c_2\lambda+c_1=0$, whose real parts are negative, whereas
\[ \lambda_1(\xi)= - \frac{c_2\,b^2}{c_1}\,\xii^2  + \textit{o}(\xii^2), \]
as~$\xi\to0$. It is not difficult to extend Theorem~\ref{thm:asymp} and show that the asymptotic profile of the solution to~\eqref{eq:CPlin}, with~$Q$ as in this example, is described by~$M/c_1$ times the fundamental solution to the heat equation
\[ c_1 v_t -c_2\,b^2\,\Delta v=0,\]
where
\[ M=\int_{\R^n} \big(u_3(x)+c_3u_2(x)+c_2u_1(x)+c_1u_0(x)\big)\,dx. \]
More precisely, the estimate
\[ \big\|\partial_t^k \big( u(t,\cdot) - v(t,\cdot)\big)\big\|_{\dot H^s} = \textit{o}\big((1+t)^{-\frac{n}4-\frac{s}2-k}\big), \]
holds.
\end{example}



\section{The influence of weak interlacing polynomials}\label{sec:weak}

In the first part of this section, we discuss what happens if the assumption of strict interlacing of the polynomials $P_{m-2}(\lambda,\xi')$ and $P_{m-1}(\lambda,\xi')$, as in~\eqref{eq:interlace}, is weakened to non strict interlacing, as in~\eqref{eq:weakinterlace}, for some~$\xi'\in S^{n-1}$. In the second part of this section, we carry on the same analysis for the interlacing of the polynomials $P_{m-1}(\lambda,\xi')$ and $P_{m}(\lambda,\xi')$. We use the same notation introduced in \textsection~\ref{sec:asymp}.

We stress that dropping the strict interlacing of the polynomials, we may also consider the case in which $P_m(\partial_t,\partial_x)$ and $P_{m-2}(\partial_t,\partial_x)$ are weakly hyperbolic. However, we cannot drop the assumption of strict hyperbolicity of $P_{m-1}(\partial_t,\partial_x)$, otherwise we lose the strict stability of $Q_2(\lambda,i\xi)$, see Lemma~\ref{lem:stable2}. More precisely, $P_m(\lambda,\xi')$ and $P_{m-2}(\lambda,\xi')$ may admit double roots for some $\xi'\in S^{n-1}$, since roots with multiplicity larger than two are prevented by the weak interlacing assumption with the strictly hyperbolic polynomial~$P_{m-1}(\lambda,\xi')$. However, a double root of $P_{m}(\lambda,\xi')$ or of $P_{m-2}(\lambda,\xi')$ is also a (simple) root of $P_{m-1}(\lambda,\xi')$, as a consequence of the interlacing condition:
\begin{align*}
a_j(\xi')=a_{j+1}(\xi') & \Rightarrow a_j(\xi')=b_j(\xi')=a_{j+1}(\xi'),\\
d_j(\xi')=d_{j+1}(\xi') & \Rightarrow d_j(\xi')=b_{j+1}(\xi')=d_{j+1}(\xi').
\end{align*}
If we allow the non strict interlacing of the polynomials $P_{m-2}(\lambda,\xi')$ and $P_{m-1}(\lambda,\xi')$ for some $\xi'\in S^{n-1}$, then Lemma~\ref{lem:low} for $Q=Q_2$ is generalized as follows.
\begin{lemma}\label{lem:low2weak}
Let~$P_{m-2}(\partial_t,\partial_x)$ and $P_m(\partial_t,\partial_x)$ be (possibly weakly) hyperbolic and $P_{m-1}(\partial_t,\partial_x)$ be strictly hyperbolic. Assume that $P_{m-2}(\lambda,\xi')$ and~$P_{m-1}(\lambda,\xi')$ interlace and that $P_{m-1}(\lambda,\xi')$ and~$P_m(\lambda,\xi')$ interlace, for any $\xi'\in S^{n-1}$, and that there is no $(\lambda,\xi')\in\R\times S^{n-1}$ such that $P_{m-2}(\lambda,\xi')=P_{m-1}(\lambda,\xi')=P_m(\lambda,\xi')$. Let~$d_j(\xi')$ be the roots of $P_{m-2}(\lambda,\xi')$, $\xi'\in S^{n-1}$.

Then we may label the $m$ roots~$\lambda_j(\xi)$, $\xi\in\R^n$, of~$Q_2(\lambda,i\xi)$, in such a way that $\lambda_{m-1}$ and $\lambda_m$ are as in~\eqref{eq:lambdam-1m}, and $\lambda_j(\xi)$, $j=1,\ldots,m-2$, are described as follows, where we put $\xi'=\xi/\xii$.
\begin{enumerate}[(i)]
\item If $P_{m-1}(d_j(\xi'),\xi')\neq0$ and $\check{P}_{m-2,j}(d_j(\xi'),\xi')\neq0$, that is, $d_j(\xi')$ is a simple root of $P_{m-2}(\lambda,\xi')$ and it is not a root of $P_{m-1}(\lambda,\xi')$, then $\lambda_j(\xi)$ is as in~\eqref{eq:lambdalow2};
\item\label{en:lows} If $P_{m-1}(d_j(\xi'),\xi')=0$ and $\check{P}_{m-2,j}(d_j(\xi'),\xi')\neq0$, that is, $d_j(\xi')$ is a simple root of both $P_{m-1}(\lambda,\xi')$ and $P_{m-2}(\lambda,\xi')$, then
\begin{equation}\label{eq:asymplows}
\lambda_j(\xi) = i\xii d_j(\xi') + i\,\xii^3\,\frac{P_m(d_j(\xi'),\xi')}{\check{P}_{m-2,j}(d_j(\xi'),\xi')} + \xii^4\,\frac{P_m(d_j(\xi'),\xi')\,\tilde P_{m-1}(d_j(\xi'),\xi')}{\big(\check{P}_{m-2,j}(d_j(\xi'),\xi')\big)^2} + \textit{o}(\xii^4),
\end{equation}
as $\xi\to0$, where
\begin{equation}\label{eq:Ptilde}
\tilde P_{m-1}(\lambda,\xi')=\begin{cases}
\check{P}_{m-1,j}(\lambda,\xi') & \text{if $d_j(\xi')=b_j(\xi')$,}\\
\check{P}_{m-1,j+1}(\lambda,\xi') & \text{if $d_j(\xi')=b_{j+1}(\xi')$.}
\end{cases}
\end{equation}
\item\label{en:lowd} If $d_j(\xi')=b_{j+1}(\xi')=d_{j+1}(\xi')$, that is, $d_j(\xi')$ is a simple root of $P_{m-1}(\lambda,\xi')$ and a double root of $P_{m-2}(\lambda,\xi')$, then
\begin{equation}\label{eq:asymplowd}
\lambda_{j,j+1}(\xi) = i\xii d_j(\xi') + \xii^2\,\kappa_\pm(\xi') + \textit{o}(\xii^2),
\end{equation}
as $\xi\to0$, where $\kappa_\pm(\xi')$ are the two solutions to
\begin{equation}\label{eq:kappalow}
\kappa^2\check{P}_{m-2,j,j+1}(d_j(\xi'),\xi') - \kappa\,\check{P}_{m-1,j+1}(d_j(\xi'),\xi') + P_m(d_j(\xi'),\xi').
\end{equation}
%
%
%
\end{enumerate}
\end{lemma}
\begin{remark}
We stress that $P_m(d_j(\xi'),\xi')\,\tilde P_{m-1}(d_j(\xi'),\xi')\neq0$ in~\eqref{eq:asymplows}, since $d_j(\xi')$ is a simple root of $P_{m-1}(\lambda,\xi')$ and $P_{m-2}(\lambda,\xi')$, therefore, it cannot be also a root of $P_m(\lambda,\xi')$. Due to the interlacing condition,
\[ \frac{P_m(d_j(\xi'),\xi')\,\tilde P_{m-1}(d_j(\xi'),\xi')}{\big(\check{P}_{m-2,j}(d_j(\xi'),\xi')\big)^2}<0. \]
We stress that $\check{P}_{m-2,j,j+1}(d_j(\xi'),\xi')\neq0$ and $\check{P}_{m-1,j+1}(d_j(\xi'),\xi')\neq0$ in~\eqref{eq:asymplowd}, since $d_j(\xi')$ is a simple root of $P_{m-1}(\lambda,\xi')$ and a double root of $P_{m-2}(\lambda,\xi')$, therefore, it cannot be also a root of $P_m(\lambda,\xi')$. Moreover,
\[ \frac{\check{P}_{m-1,j+1}(d_j(\xi'),\xi')}{2\check{P}_{m-2,j,j+1}(d_j(\xi'),\xi')}<0, \qquad \check{P}_{m-2,j,j+1}(d_j(\xi'),\xi')\,P_m(d_j(\xi'),\xi')>0, \]
as a consequence of the interlacing condition; in particular, it follows that $\Re\kappa_\pm(\xi')<0$. 
\end{remark}
\begin{proof}[Proof of Lemma~\ref{lem:low2weak}]
We follow the proof of Lemma~\ref{lem:low}, but we then take into account of the two new scenarios, \eqref{en:lows} and~\eqref{en:lowd}. In the both cases, $d_j(\xi')$ is a root of $P_{m-1}(\lambda,\xi')$. Since this latter is strictly hyperbolic, $\tilde{P}_{m-1}(d_j(\xi'),\xi')\neq0$. We first consider case~\eqref{en:lows}. 
In this case, by
\begin{equation}\label{eq:mainlow2}
0 = \xii^{-(m-2)}\,Q_2(\lambda_j,i\xi)= P_{m-2}(\mu_j,i\xi') + \rho\,P_{m-1}(\mu_j,i\xi') + \rho^2\,P_m(\mu_j,i\xi'),
\end{equation}
we obtain:
\begin{align*}
0
    & = P_{m-2}(\mu_j,i\xi') + \rho\,P_{m-1}(\mu_j,i\xi') + \rho^2\,P_m(\mu_j,i\xi') \\
    & = (\mu_j-id_j(\xi')) \, \check{P}_{m-2,j}(\mu_j,i\xi') + \rho\,(\mu_j-id_j(\xi'))\,\tilde P_{m-1}(\mu_j,i\xi') + \rho^2\,P_m(\mu_j,i\xi')\\
    & = (\mu_j-id_j(\xi')) \, \big( \check{P}_{m-2,j}(\mu_j,i\xi') + \rho\,\tilde P_{m-1}(\mu_j,i\xi')\big) + \rho^2\,P_m(\mu_j,i\xi').
\end{align*}
For sufficiently small $\rho$, $\check{P}_{m-2,j}(\mu_j,i\xi') + \rho\,\tilde P_{m-1}(\mu_j,i\xi')\neq0$, since $\check{P}_{m-2,j}(d_j(\xi'),\xi')\neq0$, therefore
\[ \mu_j-id_j(\xi') = - \frac{\rho^2\,P_m(\mu_j,i\xi')}{\check{P}_{m-2,j}(\mu_j,i\xi') + \rho\,\tilde P_{m-1}(\mu_j,i\xi')}. \]
The second-order expansion of the last expression gives:
\[ \mu_j - id_j(\xi') = i\,\rho^2\,\frac{P_m(d_j(\xi'),\xi')}{\check{P}_{m-2,j}(d_j(\xi'),\xi')} + \textit{o}(\rho^2), \]
where we used that $P_m$ is homogeneous of degree~$m$ and~$\check{P}_{m-2,j}$ is homogeneous of degree~$m-3$ to remove the imaginary unit, with a multiplication by $-i$. Since we are looking for the real part of $\mu_j$ we need one more expansion step. Writing
\[ -\frac{\rho^2\,P_m(\mu_j,i\xi')}{\check{P}_{m-2,j}(\mu_j,i\xi') + \rho\,\tilde P_{m-1}(\mu_j,i\xi')} = -\frac{\rho^2\,P_m(\mu_j,i\xi')\,\big(\check{P}_{m-2,j}(\mu_j,i\xi') - \rho\,\tilde P_{m-1}(\mu_j,i\xi')\big)}{\big(\check{P}_{m-2,j}(\mu_j,i\xi')\big)^2 - \rho^2\,\big(\tilde P_{m-1}(\mu_j,i\xi')\big)^2} \,,\]
and using the previously obtained information that $\mu_j=id_j(\xi')+\textit{O}(\rho^2)$, we easily obtain the further expansion:
\begin{align*}
\mu_j - id_j(\xi')
    & = i\,\rho^2\,\frac{P_m(d_j(\xi'),\xi')}{\check{P}_{m-2,j}(d_j,\xi')} + \rho^3\,\frac{P_m(d_j(\xi'),\xi')\,\tilde P_{m-1}(d_j(\xi'),\xi')}{\big(\check{P}_{m-2,j}(d_j(\xi'),\xi')\big)^2} + \textit{o}(\rho^3),
\end{align*}
where in the last computation we used that $P_m\,\tilde P_{m-1}$ is homogeneous of degree $2m-2$ and $(\check{P}_{m-2,j})^2$ is homogeneous of degree $2m-6$ to remove the imaginary unity with no sign change. Multiplying by $\xii$, we prove~\eqref{eq:asymplows}.

We now consider case~\eqref{en:lowd}. In this case, by~\eqref{eq:mainlow2} we obtain:
\[ (\mu_j-id_j(\xi'))^2 \, \check{P}_{m-2,j,j+1}(\mu_j,i\xi') + \rho\,(\mu_j-id_j(\xi'))\,\tilde P_{m-1}(\mu_j,i\xi') + \rho^2\,P_m(\mu_j,i\xi').\]
Noticing that $\check{P}_{m-2,j,j+1}$ is homogeneous of degree $m-4$, $\tilde P_{m-1}$ is homogeneous of degree $m-2$ and $P_m$ is homogeneous of degree~$m$, due to $\mu_j=id_j(\xi')+\textit{o}(1)$, we find
\[ \mu_j-id_j(\xi') = \rho\,\kappa_\pm(\xi')+\textit{o}(\rho),\]
for sufficiently small $\rho$, where $\kappa_\pm$ are the two solutions to~\eqref{eq:kappalow}. Therefore,
\[ \mu_{j+1}-id_j(\xi') = \rho\,\kappa_\mp(\xi')+\textit{o}(\rho). \]
Multiplying by $\xii$, we prove~\eqref{eq:asymplowd}.
\end{proof}
If we allow the non strict interlacing of the polynomials $P_{m-1}(\lambda,\xi')$ and $P_{m}(\lambda,\xi')$ for some $\xi'\in S^{n-1}$, then Lemma~\ref{lem:high} is generalized as follows.
\begin{lemma}\label{lem:highweak}
Let~$P_{m-2}(\partial_t,\partial_x)$ and $P_m(\partial_t,\partial_x)$ be (possibly weakly) hyperbolic and $P_{m-1}(\partial_t,\partial_x)$ be strictly hyperbolic. Assume that $P_{m-2}(\lambda,\xi')$ and~$P_{m-1}(\lambda,\xi')$ interlace and that $P_{m-1}(\lambda,\xi')$ and~$P_m(\lambda,\xi')$ interlace, for any $\xi'\in S^{n-1}$, and that there is no $(\lambda,\xi')\in\R\times S^{n-1}$ such that $P_{m-2}(\lambda,\xi')=P_{m-1}(\lambda,\xi')=P_m(\lambda,\xi')$. Let~$a_j(\xi')$ be the roots of $P_{m}(\lambda,\xi')$, $\xi'\in S^{n-1}$. Then we may label the $m$ roots~$\lambda_j(\xi)$ of~$Q_2(\lambda,i\xi)$, $\xi\in\R^n$, in such a way that they are described as follows, where we put $\xi'=\xi/\xii$.
\begin{enumerate}[(i)]
\item If $P_{m-1}(a_j(\xi'),\xi')\neq0$ and $\check{P}_{m,j}(a_j(\xi'),\xi')\neq0$, that is, $a_j(\xi')$ is a simple root of $P_m(\lambda,\xi')$ and it is not a root of $P_{m-1}(\lambda,\xi')$, then $\lambda_j(\xi)$ is as in~\eqref{eq:lambdahighsimple};
\item\label{en:highs} If $P_{m-1}(a_j(\xi'),\xi')=0$ and $\check{P}_{m,j}(a_j(\xi'),\xi')\neq0$, that is, $a_j(\xi')$ is a simple root of both $P_{m-1}(\lambda,\xi')$ and $P_m(\lambda,\xi')=0$, then
\begin{equation}\label{eq:asymphighs}
\lambda_j(\xi) = i\xii a_j(\xi') + i\,\xii^{-1}\,\frac{P_{m-2}(a_j(\xi'),\xi')}{\check{P}_{m,j}(a_j,\xi')} - \xii^{-2}\,\frac{P_{m-2}(a_j(\xi'),\xi')\,\tilde P_{m-1}(a_j(\xi'),\xi')}{\big(\check{P}_{m,j}(a_j(\xi'),\xi')\big)^2} + \textit{o}(\xii^{-2}),
\end{equation}
as $\xii\to\infty$, where
\[ \tilde P_{m-1}(\lambda,\xi')=\begin{cases}
\check{P}_{m-1,j-1}(\lambda,\xi') & \text{if $a_j(\xi')=b_{j-1}(\xi')$,}\\
\check{P}_{m-1,j}(\lambda,\xi') & \text{if $a_j(\xi')=b_{j}(\xi')$.}
\end{cases} \]
\item\label{en:highd} If $a_j(\xi')=b_j(\xi')=a_{j+1}(\xi')$, that is, $a_j(\xi')$ is a simple root of $P_{m-1}(\lambda,\xi')$ and a double root of $P_m(\lambda,\xi')$, then
\begin{equation}\label{eq:asymphighd}
\lambda_{j,j+1}(\xi) = i\xii a_j(\xi') + \kappa_\pm(\xi') + \textit{o}(1),
\end{equation}
as $\xii\to\infty$, where $\kappa_\pm(\xi')$ are the two solutions to
\begin{equation}\label{eq:kappahigh}
\kappa^2\check{P}_{m,j,j+1}(a_j(\xi'),\xi') + \kappa\,\check{P}_{m-1,j}(a_j(\xi'),\xi') + P_{m-2}(a_j(\xi'),\xi').
\end{equation}
%
%
%
\end{enumerate}
\end{lemma}
\begin{remark}
We stress that $P_{m-2}(a_j(\xi'),\xi')\,\tilde P_{m-1}(a_j(\xi'),\xi')\neq0$ in~\eqref{eq:asymphighs}, since $a_j(\xi')$ is a simple root of $P_{m-1}(\lambda,\xi')$ and $P_{m}(\lambda,\xi')$, therefore, it cannot be also a root of $P_{m-2}(\lambda,\xi')$. Due to the interlacing condition,
\[ -\frac{P_{m-2}(a_j(\xi'),\xi')\,\tilde P_{m-1}(a_j(\xi'),\xi')}{\big(\check{P}_{m,j}(a_j(\xi'),\xi')\big)^2}<0. \]
We stress that $\check{P}_{m,j,j+1}(a_j(\xi'),\xi')\neq0$ and $\check{P}_{m-1,j}(a_j(\xi'),\xi')\neq0$ in~\eqref{eq:asymphighd}, since $a_j(\xi')$ is a simple root of $P_{m-1}(\lambda,\xi')$ and a double root of $P_{m}(\lambda,\xi')$, therefore, it cannot be also a root of $P_{m-2}(\lambda,\xi')$. Moreover,
\[ -\frac{\check{P}_{m-1,j}(a_j(\xi'),\xi')}{2\check{P}_{m,j,j+1}(a_j(\xi'),\xi')}<0, \qquad \check{P}_{m,j,j+1}(a_j(\xi'),\xi')\,P_{m-2}(a_j(\xi'),\xi')>0, \]
as a consequence of the interlacing condition; in particular, it follows that $\Re\kappa_\pm(\xi')<0$. 
\end{remark}
\begin{proof}[Proof of Lemma~\ref{lem:highweak}]
The proof is similar to the proof of Lemma~\ref{lem:low2weak}. We follow the proof of Lemma~\ref{lem:high}, but we then take into account of the two new scenarios, \eqref{en:highs} and~\eqref{en:highd}. In the both cases, $a_j(\xi')$ is a root of $P_{m-1}(\lambda,\xi')$. Since this latter is strictly hyperbolic, $\tilde{P}_{m-1}(a_j(\xi'),\xi')\neq0$. We first consider case~\eqref{en:highs}. In this case, by
\begin{equation}\label{eq:mainhigh2}
0 = \xii^{-m}\,Q_2(\lambda_j,i\xi)= P_m(\mu_j,i\xi') + \rho^{-1}\,P_{m-1}(\mu_j,i\xi') + \rho^{-2}\,P_{m-2}(\mu_j,i\xi'),
\end{equation}
we find
\begin{align*}
0
    & = (\mu_j-ia_j(\xi')) \check{P}_{m,j}(\mu_j,i\xi') + \rho^{-1}\,(\mu_j-ia_j(\xi'))\,\tilde{P}_{m-1}(\mu_j,i\xi') + \rho^{-2}\,P_{m-2}(\mu_j,i\xi')\\
    & = (\mu_j-ia_j(\xi'))\big(\check{P}_{m,j}(\mu_j,i\xi') + \rho^{-1}\,\tilde{P}_{m-1}(\mu_j,i\xi')\big)+ \rho^{-2}\,P_{m-2}(\mu_j,i\xi').
\end{align*}
For sufficiently large $\rho$, $\check{P}_{m,j}(\mu_j,i\xi') + \rho^{-1}\,\tilde P_{m-1}(\mu_j,i\xi')\neq0$, since $\check{P}_{m,j}(a_j(\xi'),\xi')\neq0$, therefore
\[ \mu_j-ia_j(\xi') = -\frac{\rho^{-2}\,P_{m-2}(\mu_j,i\xi')}{\check{P}_{m,j}(\mu_j,i\xi') + \rho^{-1}\,\tilde P_{m-1}(\mu_j,i\xi')}. \]
The second-order expansion of the last expression gives:
\[ \mu_j - ia_j(\xi') = i\,\rho^{-2}\,\frac{P_{m-2}(a_j(\xi'),\xi')}{\check{P}_{m,j}(a_j,\xi')} + \textit{o}(\rho^{-2}), \]
where we used that $P_{m-2}$ is homogeneous of degree~$m-2$ and~$\check{P}_{m,j}$ is homogeneous of degree~$m-1$ to remove the imaginary unit, with a multiplication by $-i$. Since we are looking for the real part of $\mu_j$ we need one more expansion step. Writing
\[ -\frac{\rho^{-2}\,P_{m-2}(\mu_j,i\xi')}{\check{P}_{m,j}(\mu_j,i\xi') + \rho^{-1}\,\tilde P_{m-1}(\mu_j,i\xi')} = - \frac{\rho^{-2}\,P_{m-2}(\mu_j,i\xi')\,\big(\check{P}_{m,j}(\mu_j,i\xi') - \rho^{-1}\,\tilde P_{m-1}(\mu_j,i\xi')\big)}{\big(\check{P}_{m,j}(\mu_j,i\xi')\big)^2 - \rho^{-2}\,\big(\tilde P_{m-1}(\mu_j,i\xi')\big)^2} \,,\]
and using the previously obtained information that $\mu_j=ia_j(\xi')+\textit{O}(\rho^{-2})$, we easily obtain the further expansion:
\begin{align*}
\mu_j - ia_j(\xi')
    & = i\,\rho^{-2}\,\frac{P_{m-2}(a_j(\xi'),\xi')}{\check{P}_{m,j}(a_j,\xi')} - \rho^{-3}\,\frac{P_{m-2}(a_j(\xi'),\xi')\,\tilde P_{m-1}(a_j(\xi'),\xi')}{\big(\check{P}_{m,j}(a_j(\xi'),\xi')\big)^2} + \textit{o}(\rho^{-3}),
\end{align*}
where in the last computation we used that $P_{m-2}\,\tilde P_{m-1}$ is homogeneous of degree $2m-4$ and $(\check{P}_{m,j})^2$ is homogeneous of degree $2m-2$ to remove the imaginary unity with a sign change. Multiplying by $\xii$, we prove~\eqref{eq:asymphighs}.

We now consider case~\eqref{en:highd}. In this case, by~\eqref{eq:mainhigh2} we obtain:
\[ (\mu_j-ia_j(\xi'))^2 \, \check{P}_{m,j,j+1}(\mu_j,i\xi') + \rho^{-1}\,(\mu_j-ia_j(\xi'))\,\check{P}_{m-1,j}(\mu_j,i\xi') + \rho^{-2}\,P_{m-2}(\mu_j,i\xi').\]
Noticing that $\check{P}_{m,j,j+1}$, $\check{P}_{m-1,j}$ and $P_{m-2}$ are homogeneous of degree $m-2$, due to $\mu_j=ia_j(\xi')+\textit{o}(1)$, we find
\[ \mu_j-ia_j(\xi') = \rho^{-1}\,\kappa_\pm(\xi')+\textit{o}(\rho^{-1}),\]
for sufficiently large $\rho$, where $\kappa_\pm$ are the two solutions to~\eqref{eq:kappahigh}. Consequently,
\[ \mu_{j+1}-ia_j(\xi') = \rho^{-1}\,\kappa_\mp(\xi')+\textit{o}(\rho^{-1}). \]
Multiplying by $\xii$, we prove~\eqref{eq:asymphighd}.
\end{proof}
To discuss the influence on the decay estimates for the solution to~\eqref{eq:CPlin}, coming from the four scenarios in Lemmas~\ref{lem:low2weak} and~\ref{lem:highweak}, it is convenient to localize $u$ at low and high frequencies to treat separately the scenarios at low and high frequencies. Without loss of generality, this localization may be more easily expressed, localizing the initial data.
\begin{theorem}\label{thm:lowloss}
Assume that Hypothesis~\ref{hyp:Q2} holds, and that
\[ u_j\in L^q \cap L^2 \qquad \text{for some $q\in[1,2]$, with $\hat u_j$ compactly supported, for $j=0,\ldots,m-1$.} \]
Then the solution $u$ to Cauchy problem~\eqref{eq:CPlin} with $Q=Q_2$ satisfies the following estimates:
\begin{equation}
\label{eq:estQ2weakgen}\begin{split}
\|\partial_t^ku(t,\cdot)\|_{\dot H^s}
    & \leq C\,\sum_{j=m-3}^{m-1}(1+t)^{-\eta}\,\big(\|u_j\|_{L^q}+\|u_j\|_{L^2}\big)\\
    & \qquad +C\,\sum_{j=0}^{m-4}(1+t)^{-\frac{n}4\left(\frac1q-\frac12\right)-\frac{k+s-j}4}\,\big(\|u_j\|_{L^q}+\|u_j\|_{L^2}\big),
\end{split}\end{equation}
where~$\eta$ is as in~\eqref{eq:estQ2worst}, 
provided that $k+s\geq m-3$ if $q=2$, or $n(1/q-1/2)+k+s>m-3$ otherwise. The decay in~\eqref{eq:estQ2worst} may be improved in two cases.
\begin{enumerate}[(i)]
\item\label{en:strictlowloss} If $P_{m-2}(\partial_t,\partial_x)$ is strictly hyperbolic, then~\eqref{eq:estQ2weakgen} may be improved to
\begin{equation}
\label{eq:estQ2strict}\begin{split}
\|\partial_t^ku(t,\cdot)\|_{\dot H^s}
     & \leq C\,\sum_{j=m-3}^{m-1}(1+t)^{-\frac{n}4\left(\frac1q-\frac12\right)+\frac{k+s-(m-3)}4}\,\big(\|u_j\|_{L^q}+\|u_j\|_{L^2}\big) \\
     & \qquad + C\,\sum_{j=0}^{m-4}(1+t)^{-\frac{n}4\left(\frac1q-\frac12\right)+\frac{k+s-j}4}\,\big(\|u_j\|_{L^q}+\|u_j\|_{L^2}\big).
\end{split}\end{equation}
\item\label{en:weaklowloss} Assume that for any $\xi'\in S^{n-1}$, the interlacing of $P_{m-2}(\lambda,\xi')$ and $P_{m-1}(\lambda,\xi')$ is strict, exception given for the double roots of $P_{m-2}(\lambda,\xi')$, that is, $b_1(\xi')<d_1(\xi')$, $d_{m-2}(\xi')<b_{m-1}(\xi')$, and
\[ d_j(\xi')<d_{j+1}(\xi') \Rightarrow d_j(\xi')<b_{j+1}(\xi')<d_{j+1}(\xi'). \]
Then~\eqref{eq:estQ2weakgen} may be improved to
\begin{equation}
\label{eq:estQ2strong}\begin{split}
\|\partial_t^ku(t,\cdot)\|_{\dot H^s}
     & \leq C\,\sum_{j=m-3}^{m-1}(1+t)^{-\frac{n}2\left(\frac1q-\frac12\right)-\frac{k+s-(m-2)}2}\,\big(\|u_j\|_{L^q}+\|u_j\|_{L^2}\big) \\
     & \qquad + C\,\sum_{j=0}^{m-4}(1+t)^{-\frac{n}2\left(\frac1q-\frac12\right)-\frac{k+s-j-1}2}\,\big(\|u_j\|_{L^q}+\|u_j\|_{L^2}\big).
\end{split}
\end{equation}
\end{enumerate}
\end{theorem}
\begin{proof}
We follow the proof of Theorem~\ref{thm:CPQ2}, but we take into account of the differences in the behavior of $\lambda_j(\xi)$ as $\xi\to0$. Let $\xii\leq\delta$ for sufficiently small $\delta>0$. For $j=m-1,m$, we may proceed as in the proof of Theorem~\ref{thm:CPQ2} and derive
\[ \frac{\big|\hat u_{m-1}-\hat u_{m-2}\sum_{k\neq j}\lambda_k + \ldots + (-1)^{m-1} \hat u_0 \prod_{k\neq j}\lambda_k\big|}{\prod_{k\neq j}|\lambda_k-\lambda_j|}\,e^{\lambda_jt} \lesssim e^{-ct}\,\sum_{k=0}^{m-1} |\hat u_k|, \]
if $4c_{m-2,0}\neq c_{m-1,0}^2$, and similarly if the equality holds.

For $j=1,\ldots,m-2$, for any~$\xi'\in\mathcal S^{n-1}$, we consider the three different scenarios in Lemma~\ref{lem:low2weak}.

If $P_{m-1}(d_j(\xi'),\xi')\neq0$, 
then $\lambda_j(\xi)$ is as in~\eqref{eq:lambdalow2}, and we get the pointwise estimate
\begin{align*}
& \frac{\big|\hat u_{m-1}-\hat u_{m-2}\sum_{k\neq j}\lambda_k + \ldots + (-1)^{m-1} \hat u_0 \prod_{k\neq j}\lambda_k\big|}{\prod_{k\neq j}|\lambda_k-\lambda_j|}\,e^{\lambda_jt} \\
    & \qquad \leq e^{-ct\xii^2}\, \left( \xii^{-(m-3)}\big(|\hat u_{m-1}|+|\hat u_{m-2}|+|\hat u_{m-3}|\big) + \sum_{k=0}^{m-4}\xii^{-j}\,|\hat u_k|\right),
\end{align*}
as in the proof of Theorem~\ref{thm:CPQ2}.

If $P_{m-1}(d_j(\xi'),\xi')=0$ and $\check{P}_{m-2,j}(d_j(\xi'),\xi')\neq0$, that is, $d_j(\xi')$ is a simple root of both $P_{m-1}(\lambda,\xi')$ and $P_{m-2}(\lambda,\xi')$, then $\lambda_j(\xi)$ is as in~\eqref{eq:asymplows}, and we get the pointwise estimate
\begin{align*}
& \frac{\big|\hat u_{m-1}-\hat u_{m-2}\sum_{k\neq j}\lambda_k + \ldots + (-1)^{m-1} \hat u_0 \prod_{k\neq j}\lambda_k\big|}{\prod_{k\neq j}|\lambda_k-\lambda_j|}\,e^{\lambda_jt} \\
    & \qquad \leq e^{-ct\xii^4}\, \left( \xii^{-(m-3)}\big(|\hat u_{m-1}|+|\hat u_{m-2}|+|\hat u_{m-3}|\big) + \sum_{k=0}^{m-4}\xii^{-j}\,|\hat u_k|\right),
\end{align*}
where the only difference with the case treated in Theorem~\ref{thm:CPQ2} is that the exponential term $e^{-ct\xii^2}$ is replaced by $e^{-ct\xii^4}$. Indeed, it remains valid that
\begin{align*}
& |\lambda_k|\lesssim\xii, \qquad k=1,\ldots,m-2,\\
& |\lambda_k-\lambda_j|=|d_k(\xi')-d_j(\xi')|\,\xii\gtrsim \xii, \quad k=1,\ldots,m-2, \ k\neq j,
\end{align*}
since $d_j(\xi')$ is a simple root of $P_{m-2}(\lambda,\xi')$. In particular, if $P_{m-2}(\lambda,\xi')$ is strictly hyperbolic for any $\xi'\in S^{n-1}$, following as in the proof of Theorem~\ref{thm:CPQ2}, this just leads to replace the decay rate in~\eqref{eq:estQ2} with the decay rate in~\eqref{eq:estQ2strict}.

If $d_j(\xi')$ is a simple root of $P_{m-1}(\lambda,\xi')$ and a double root of $P_{m-2}(\lambda,\xi')$, then $\lambda_j(\xi)$ is as in~\eqref{eq:asymplowd}, in particular, $\Re\lambda_j(\xi)\leq-c\xii^2$, so that the power at the denominator in the decay rate remains as in~\eqref{eq:estQ2}. However, in this case, $\lambda_j(\xi)-\lambda_{j+1}(\xi)$ vanishes at a faster speed as~$\xi\to0$. Rewriting the representation of $\hat u(t,\xi)$ as in Remark~\ref{rem:multiple}, we get the pointwise estimate
\begin{align*}
& \frac{\big|\hat u_{m-1}-\hat u_{m-2}\sum_{k\neq j}\lambda_k + \ldots + (-1)^{m-1} \hat u_0 \prod_{k\neq j}\lambda_k\big|}{\prod_{k\neq j}|\lambda_k-\lambda_j|}\,e^{\lambda_jt} \\
    & \qquad \leq (1+t\xii)\,e^{-ct\xii^2}\, \left( \xii^{-(m-3)}\big(|\hat u_{m-1}|+|\hat u_{m-2}|+|\hat u_{m-3}|\big) + \sum_{k=0}^{m-4}\xii^{-j}\,|\hat u_k|\right),
\end{align*}
where the only difference with the case treated in Theorem~\ref{thm:CPQ2} is the term $(1+t\xii)$; this term produces a loss of decay rate $\sqrt{t}$ for large $t$. In turn, if the interlacing of $P_{m-2}(\lambda,\xi')$ and $P_{m-1}(\lambda,\xi')$ is strict, exception given for the double roots of $P_{m-2}(\lambda,\xi')$, for any $\xi'\in S^{n-1}$, following as in the proof of Theorem~\ref{thm:CPQ2}, this just leads to replace the decay rate in~\eqref{eq:estQ2} with the decay rate in~\eqref{eq:estQ2strong}, where the loss of decay $(1+t)^{\frac12}$ appears.

In the general case, if $P_{m-2}(\lambda,\xi')$ is not strictly hyperbolic, and the interlacing of $P_{m-2}(\lambda,\xi')$ and $P_{m-1}(\lambda,\xi')$ is weak, for some $\xi'\in S^{n-1}$, either of the previous scenario may happen for different roots $\lambda_j$ and/or at different points $\xi'\in S^{n-1}$, so that we shall consider the worst case scenario, that is, we shall replace the decay rate in~\eqref{eq:estQ2} with the one in~\eqref{eq:estQ2worst}.
\end{proof}
At high frequencies, we shall distinguish two cases. Indeed, if $P_m(\partial_t,\partial_x)$ is hyperbolic, but not strictly, then~\eqref{eq:CPlin} is not well-posed in $H^s$, in general, but only in $H^\infty$ (or $\mathcal C^\infty$), since a loss of $1$ derivative may occur. For this reason, we first discuss the case in which $P_m(\partial_t,\partial_x)$ is strictly hyperbolic, but the interlacing of $P_{m-1}(\lambda,\xi')$ and $P_m(\lambda,\xi')$ may be not strict for some $\xi'\in S^{n-1}$.
\begin{proposition}\label{prop:lossdecay}
Assume that Hypothesis~\ref{hyp:Q2} holds, and that $P_m(\partial_t,\partial_x)$ is strictly hyperbolic. Assume that
\[ u_j\in H^{s_0-j} \qquad \text{with $\hat u_j=0$ in a neighborhood of the origin, $j=0,\ldots,m-1$,} \]
for some $s_0\geq m-1$. Then the solution $u$ to Cauchy problem~\eqref{eq:CPlin} with $Q=Q_2$ satisfies the following regularity-loss type decay estimate:
\begin{equation}
\label{eq:estQ2loss}
\|\partial_t^ku(t,\cdot)\|_{\dot H^s} \leq C\,(1+t)^{-\frac\nu2}\,\|u_j\|_{H^{k+s+\nu-j}},
\end{equation}
for any~$\nu\geq0$ such that $k+s+\nu\leq s_0$. 
If, moreover, $P_{m-1}(\lambda,\xi')$ and $P_m(\lambda,\xi')$ strictly interlace for any $\xi'\in S^{n-1}$, then~\eqref{eq:estQ2loss} is improved to
\begin{equation}
\label{eq:estQ2noloss}
\|\partial_t^ku(t,\cdot)\|_{\dot H^s} \leq C\,e^{-ct}\,\|u_j\|_{H^{k+s-j}},
\end{equation}
for any $k+s\leq s_0$. 
\end{proposition}
The type of decay rate in~\eqref{eq:estQ2loss} is a \emph{regularity-loss decay,} in the sense that additional initial data regularity produces extra decay rate for the solution~\cite{HK06}. This phenomenon is known for some plate models under rotational inertia effects~\cite{CDI13p, SK10} and it also holds for wave equations with very strong damping~\cite{GGH} $\partial_t^2-\Delta u+(-\Delta)^ku_t=0$, $k>1$. Proposition~\ref{prop:lossdecay} shows that, as a consequence of Lemma~\ref{lem:highweak}, this phenomenon may also hold for pure hyperbolic equations.
\begin{proof}
We follow the proof of Theorem~\ref{thm:CPQ2} at high frequencies, but now, if $a_j(\xi')$ is a simple root of both $P_{m-1}(\lambda,\xi')$ and $P_m(\lambda,\xi')$ for some $\xi'\in S^{n-1}$ and $j=1,\ldots,m$, then, due to~\eqref{eq:asymphighs}, recalling~\eqref{eq:urepgen}, we may only estimate
\[ |\partial_t^k\hat u(t,\xi)| \lesssim \sum_{j=0}^{m-1} \xii^{k-j} |\hat u_j(\xi)|\,e^{-\frac{ct}{\xii^2}}, \]
for~$k\geq0$. We immediately obtain that, for any~$\nu\geq0$ such that $k+s+\nu\leq s_0$,
\[
\Big(\int_{\xii\geq M} \xii^{2s} |\partial_t^k\hat u(t,\xi)|^2\,d\xi\Big)^{\frac12} \lesssim \big(\sup_{\xii\geq M} \xii^{-\nu}\,e^{-\frac{ct}{\xii^2}}\big)\,\sum_{j=0}^{m-1} \|u_j\|_{H^{s+k+\nu-j}},
\]
and~\eqref{eq:estQ2loss} follows.
\end{proof}
When $P_m(\partial_t,\partial_x)$ is not strictly hyperbolic, Cauchy problem~\eqref{eq:CPlin} is well-posed in $H^\infty$.
\begin{proposition}\label{prop:regloss}
Assume that Hypothesis~\ref{hyp:Q2} holds, and that
\[ u_j\in H^\infty \qquad \text{with $\hat u_j=0$ in a neighborhood of the origin, $j=0,\ldots,m-1$,} \]
Then the solution $u$ to Cauchy problem~\eqref{eq:CPlin} with $Q=Q_2$ satisfies decay estimate~\eqref{eq:estQ2loss} for any~$\nu\geq1$. If, moreover, the interlacing of $P_{m-1}(\lambda,\xi')$ and $P_{m}(\lambda,\xi')$ is strict, exception given for the double roots of $P_{m}(\lambda,\xi')$ for any $\xi'\in S^{n-1}$, that is,
\[ a_j(\xi')<a_{j+1}(\xi') \Rightarrow a_j(\xi')<b_{j}(\xi')<a_{j+1}(\xi'), \]
then the solution $u$ to Cauchy problem~\eqref{eq:CPlin} with $Q=Q_2$ satisfies the decay estimate
\begin{equation}
\label{eq:estQ2regloss}
\|\partial_t^ku(t,\cdot)\|_{\dot H^s} \leq C\,e^{-ct}\,\|u_j\|_{H^{k+s+1-j}}.
\end{equation}
%
\end{proposition}
The loss of $1$ derivative in estimate~\eqref{eq:estQ2regloss}, and the restriction~$\nu\geq1$ in~\eqref{eq:estQ2loss} are consistent to the assumption of initial data in $H^\infty$, in order to get the well-posedness (in $H^\infty$) of Cauchy problem~\eqref{eq:CPlin}.
\begin{proof}
We follow the proof of Theorem~\ref{thm:CPQ2} at high frequencies, but now, $a_j(\xi')$ may be a double root of $P_m(\lambda,\xi')$ for some $\xi'\in S^{n-1}$. If $a_j(\xi')=a_{j+1}(\xi')$, the asymptotic behavior of $\lambda_j$ and $\lambda_{j+1}$ is described by~\eqref{eq:asymphighd}, and we may rely on Remark~\ref{rem:multiple} to estimate
\begin{align*}
& \left|\frac{\hat u_{m-1}-\hat u_{m-2}\sum_{k\neq j}\lambda_k + \ldots + (-1)^{m-1} \hat u_0 \prod_{k\neq j}\lambda_k}{\prod_{k\neq j}(\lambda_k-\lambda_j)}\,e^{\lambda_jt}\right.\\
& \qquad \left. + \frac{\hat u_{m-1}-\hat u_{m-2}\sum_{k\neq j+1}\lambda_k + \ldots + (-1)^{m-1} \hat u_0 \prod_{k\neq j+1}\lambda_k}{\prod_{k\neq j+1}(\lambda_k-\lambda_j)}\,e^{\lambda_{j+1}t}\right|\\
& \lesssim (1+t\xii)\,e^{-ct}\,\sum_{k=0}^{m-1}|\hat u_k|\,\xii^{-k}.
\end{align*}
If the interlacing of $P_{m-1}(\lambda,\xi')$ and $P_{m}(\lambda,\xi')$ is strict, exception given for the double roots of $P_{m}(\lambda,\xi')$, for any $\xi'\in S^{n-1}$, we immediately derive~\eqref{eq:estQ2regloss}. Otherwise, taking into account that $a_j(\xi')$ may be a simple root of both $P_{m-1}(\lambda,\xi')$ and $P_m(\lambda,\xi')$ for some $\xi'\in S^{n-1}$ and $j=1,\ldots,m$, we proceed as in the proof of Proposition~\ref{prop:regloss}. Considering both scenarios, we derive~\eqref{eq:estQ2loss} with~$\nu\geq1$.
\end{proof}
Theorem~\ref{thm:lowloss} and Propositions~\ref{prop:lossdecay} and~\ref{prop:regloss} may be combined in different ways, according to the interlacing assumptions of $P_{m-2}(\lambda,\xi')$ with $P_{m-1}(\lambda,\xi')$ and of $P_{m-1}(\lambda,\xi')$ with $P_{m}(\lambda,\xi')$, and to the strict or weak hyperbolicity of $P_{m-2}(\partial_t,\partial_x)$ and $P_{m}(\partial_t,\partial_x)$. In the very general case, with no extra assumption with respect to the necessary and sufficient condition for the strict stability, given in Hypothesis~\ref{hyp:Q2}, the proof of Theorem~\ref{thm:verygen} follows by combining estimate in~\eqref{eq:estQ2weakgen} in Theorem~\ref{thm:lowloss} and estimate~\eqref{eq:estQ2loss} with $\nu\geq1$ in Proposition~\ref{prop:regloss}.

An analogous result to Theorem~\ref{thm:asymp} may be obtained to describe the asymptotic behavior of the solution to~\eqref{eq:CPlin}, following the proof of Theorem~\ref{thm:asymp}, under the more general assumptions in this section. For the sake of brevity, we provide such result in the simpler anisotropic case, as in Examples \ref{ex:MGT}, \ref{ex:BC}, \ref{ex:electric}.
\begin{theorem}\label{thm:asympweak}
Assume that Hypothesis~\ref{hyp:Q2} holds, that
\[ u_j\in L^1\cap H^\infty \qquad \text{for $j=0,\ldots,m-1$,} \]
and that $P_{m-j}(\lambda,\xi')=P_{m-j}(\lambda)$ are anisotropic, i.e., independent of~$\xi'\in S^{n-1}$, $j=0,1,2,$. Let
\[ J_1=\{j: \ d_j=b_{j+1}=d_{j+1} \}, \qquad J_2=\{j: \ d_{j-1}<b_j=d_j \ \text{or} \ d_j=b_{j+1}<d_{j+1} \}, \]
and define
\begin{align}
\label{eq:vweak}
\hat v(t,\xi)
    & = M\,\sum_{j\in J_1}\frac{e^{i\xii\,d_jt}}{i^{m-4}\check{P}_{m-2,j,j+1}(d_j)}\,\frac{e^{\xii^2\kappa_-t}-e^{\xii^2\kappa_+t}}{\kappa_+-\kappa_-}\,, \\
\label{eq:wweak}
\hat w(t,\xi)
    & = M\,\sum_{j\in J_2}\frac{e^{i\xii\,d_jt + i\,\xii^3\,\frac{P_m(d_j)}{\check{P}_{m-2,j}(d_j)}\,t + \xii^4\,\frac{P_m(d_j)\,\tilde P_{m-1}(d_j)}{(\check{P}_{m-2,j}(d_j))^2}\,t}}{i^{m-3}\check{P}_{m-2,j}(d_j)}\,,
\end{align}
with $M$ as in~\eqref{eq:moment}, $\tilde P_{m-1}$ as in~\eqref{eq:Ptilde}, and $\kappa_\pm$ as in~\eqref{eq:kappalow}, provided that $\kappa_+\neq \kappa_-$. If $J_2$ is empty and $J_1$ is nonempty, then the solution $u$ to Cauchy problem~\eqref{eq:CPlin} with $Q=Q_2$ satisfies the following estimate
\begin{equation}
\label{eq:Q2asympweak2}
\big\|\partial_t^k \big( u(t,\cdot) - I_{m-2}v(t,\cdot)\big)\big\|_{\dot H^s} = \textit{o}\big((1+t)^{-\frac{n}4-\frac{k+s-(m-2)}2}\big),
\end{equation}
for $n/2+k+s>m-2$. If $J_2$ is nonempty, then the solution $u$ to Cauchy problem~\eqref{eq:CPlin} with $Q=Q_2$ satisfies estimate~\eqref{eq:Q2asympweak2} for $m-2<n/2+k+s<m-1$ and the following estimate
\begin{equation}
\label{eq:Q2asympweak1}
\big\|\partial_t^k \big( u(t,\cdot) - I_{m-3}w(t,\cdot)\big)\big\|_{\dot H^s} = \textit{o}\big((1+t)^{-\frac{n}8-\frac{k+s-(m-3)}4}\big),
\end{equation}
for $n/2+k+s>m-1$. Moreover, \eqref{eq:Q2asympweak1} holds for $n/2+k+s>m-3$ if $J_1$ is empty.
\end{theorem}
The proof is analogous to the proof of Theorem~\ref{thm:asymp}, so we omit it.

\begin{example}\label{ex:electricdiss}
We go back to Example~\ref{ex:electric}, but we add a dissipation term $a\partial_tu$ to the system of elastic waves~\eqref{eq:MHD3dplus}:
\begin{equation}\label{eq:MHD3diss}
\begin{cases}
\partial_t^2 u + a \partial_t u - \mu \Delta u - (\mu+\nu) \nabla \div u + \gamma \curl E =0,\\
\partial_t E +\sigma E - \curl H -\gamma\curl \partial_t u=0,\\
\partial_t H + c^2\curl E =0.
\end{cases}
\end{equation}
Proceeding as in Example~\ref{ex:electric}, and defining $\varphi=\partial_tE+\sigma E$ for each component of $E$, $\varphi$ satisfies the fourth order equation
\[ \big( (\partial_t^2+a\partial_t-\mu\Delta)(\partial_t^2+\sigma\partial_t-c^2\Delta)-\gamma^2\partial_t^2\Delta\big) \varphi=0. \]
The equation above is $Q_2(\partial_t,\partial_x)\varphi=0$, where $Q_2=P_4+P_3+P_2$ and
\begin{align*}
P_4(\lambda,\xi)
    & = \lambda^4-(\mu+c^2+\gamma^2)\lambda^2\xii^2+c^2\mu\xii^4, \\
P_3(\lambda,\xi)
    & = (a+\sigma)\lambda^3 - (ac^2+\mu\sigma)\lambda\xii^2, \\
P_2(\lambda,\xi)
    & = a\sigma\lambda^2.
\end{align*}
Now~$P_2(\lambda,\xi')$ has a double root~$0$, the roots of~$P_3(\lambda,\xi')$ are~$b_2=0$ and
\[ b_3=\sqrt{\frac{ac^2+\mu\sigma}{a+\sigma}}, \qquad b_1=-b_3, \]
and the roots of $P_4(\lambda,\xi')$ are
\[ a_{3,4} = \frac1{\sqrt{2}}\,\sqrt{\mu+c^2+\gamma^2\pm \sqrt{(\mu+c^2+\gamma^2)^2-4c^2\mu}}, \quad a_{1,2}=-a_{3,4}. \]
The strict interlacing condition of $P_3(\lambda,\xi')$ and $P_4(\lambda,\xi')$ may be checked by straightforward calculations. Therefore, Theorem~\ref{thm:lowloss}, \eqref{en:lowd} and Proposition~\ref{prop:lossdecay} apply. 
%
%
By Lemma~\ref{lem:low2weak}, due to
\[ \check{P}_{2,1,2}(0,\xi')=a\sigma, \qquad \check{P}_{3,2}(0,\xi')=-(ac^2+\mu\sigma), \qquad P_4(0,\xi')=c^2\mu, \]
we may compute
\[ \lambda_4(0)=-a, \qquad \lambda_3(0)=-\sigma, \qquad \lambda_2(\xi)=-\frac{ac^2}{a+\sigma}\,\xii^2+\textit{o}(\xii^2), \qquad \lambda_1(\xi)=-\frac{\mu\sigma}{a+\sigma}\,\xii^2+\textit{o}(\xii^2),\]
as $\xi\to0$. By Theorem~\ref{thm:asympweak}, if $ac^2\neq \mu\sigma$, the asymptotic profile of $\varphi(t,\cdot)$ is described by $I_2v$, where
\begin{align*}
\mathfrak{F}(I_2v)
    & = \xii^{-2}\hat v(t,\xi) = M\,\frac{a+\sigma}{a\sigma(\mu\sigma-ac^2)}\,\frac{e^{-\frac{ac^2}{a+\sigma}\,\xii^2 t}-e^{-\frac{\mu\sigma}{a+\sigma}\,\xii^2 t}}{\xii^2}\,,
\intertext{if}
M
    & = \int_{\R^3} \big(\partial_t^3\varphi(0,x)+(a+\sigma)\partial_t^2 \varphi(0,x)+a\sigma \partial_t \varphi(0,x)\big)\,dx \neq0.
\end{align*}
Explicitly,
\[ \|\partial_t^k\big( \varphi(t,\cdot)-I_2v(t,\cdot)\big)\|_{\dot H^s} =\textit{o}\big((1+t)^{-\frac34-(k-1)-\frac{s}2}\big),\] 
for~$k\geq1$ or~$s>1/2$.
\end{example}
We also provide an example of the application of Lemma~\ref{lem:low2weak} to the theory of systems of elastic waves with anisotropic dissipation.
\begin{example}\label{ex:anisotropic}
We consider the eigenvalues of a system of elastic waves in $\R^2$,
\[ \partial_t^2 u - \mu \Delta u - (\mu+\nu)\nabla \div u + A\partial_tu =0, \qquad A=\begin{pmatrix}
a_1 & 0 \\
0 & a_2
\end{pmatrix} \]
where $A$ with $a_1>a_2>0$, represents an anisotropic dissipation. The constant~$\mu,\nu$ are the Lam\`e constants and verify $\mu>0$, $\mu+\nu>0$. Let $U=(\xii \hat u, \partial_t \hat u)$. Then
\[ \partial_t U = MU, \quad M=\begin{pmatrix}
0 & \xii \mathrm{I_2}\\
-\xii^{-1}B & -A
\end{pmatrix}, \qquad B=
\begin{pmatrix}
(2\mu+\nu)\xi_1^2+\mu\xi_2^2 & (\mu+\nu)\xi_1\xi_2\\
(\mu+\nu)\xi_1\xi_2 & \mu\xi_1^2+(2\mu+\nu)\xi_2^2 \end{pmatrix}. \]
By straightforward computation,
\[ \det (\lambda-M) = P_4(\lambda,i\xi)+P_3(\lambda,i\xi)+P_2(\lambda,i\xi), \]
where
\begin{align*}
P_4(\lambda,\xi)
    & = (\lambda^2-\mu\xii^2)(\lambda^2-(2\mu+\nu)\xii^2),\\
P_3(\lambda,\xi)
    & = \lambda^3 (a_1+a_2)- \lambda \big( a_1(\mu\xii^2+(\mu+\nu)\xi_2^2) + a_2(\mu\xii^2+(\mu+\nu)\xi_1^2)\big),\\
P_2(\lambda,\xi)
    & = \lambda^2 a_1a_2.
\end{align*}
In particular, $P_4(\partial_t,\partial_x)$ and $P_3(\partial_t,\partial_x)$ are strictly hyperbolic, whereas $P_2(\partial_t,\partial_x)$ is weakly hyperbolic. The interlacing condition holds, since the roots of $P_4(\lambda,\xi)$ are $\pm \xii\sqrt{\mu}$ and $\pm \xii\sqrt{2\mu+\nu}$, whereas the roots of $P_3(\lambda,\xi)$ are $0$ and
\[ \pm \xii\,\sqrt{\mu + (\mu+\nu)\frac{a_1\xi_2^2+a_2\xi_1^2}{(a_1+a_2)\,\xii^2}}\,. \]
Therefore, the eigenvalues $\lambda_j(\xi)$ of~$M$ verify $\Re\lambda_j(\xi)<0$ for any $\xi\neq0$. Due to Lemma~\ref{lem:low2weak}, case~\eqref{en:lowd},
\[ \lambda_{4}(0)=-a_2, \quad \lambda_3(0)=-a_1, \quad \lambda_{1,2}(\xi)= \xii^2\,\kappa_\pm(\xi')+\textit{o}(\xii^2),\]
as $\xi\to0$, where $\kappa_\pm(\xi')$, $\xi'\in S^1$, are the two solutions to
\[ \kappa^2 a_1a_2 + \kappa \big( a_1(\mu+(\mu+\nu)(\xi_2')^2) + a_2(\mu+(\mu+\nu)(\xi_1')^2)\big) + \mu(2\mu+\nu) =0. \]
%
%
We stress that, as expected, in the limit case $a_2=0$, the polynomial $Q(\lambda,i\xi)=P_4(\lambda,i\xi)+P_3(\lambda,i\xi)$ is not strictly stable, since the interlacing of $P_3(\lambda,\xi)$ and $P_4(\lambda,\xi)$ is not strict at $(\pm1,0)$ and at $(0,\pm1)$.
\end{example}
Finally, we present two examples without a physical motivation, just to clarify the possible scenarios of weak interlacing from a mathematical point of view.
\begin{example}
If we replace the viscoelastic damping $-b\Delta\partial_t$ in the MGT equation in Example~\ref{ex:MGT} by a classical damping $b\partial_t$, then we obtain the third order operator $Q=Q_2(\partial_t,\partial_x)$, where
\[ P_3(\lambda,\xi')=\lambda(\lambda^2-c^2), \quad P_2(\lambda,\xi')=\tau^{-1}(\lambda^2-c^2), \quad P_1(\lambda,\xi')=\tau^{-1}b\lambda, \]
to which Lemma~\ref{lem:highweak} applies, and by~\eqref{eq:asymphighs} we derive the asymptotic behaviors
\[ \lambda_{1,3}(\xi) = \pm i c\xii \pm \frac{ib}{2c\tau\xii} - \frac{b}{2c^2\tau^2\xii^2} + \textit{o}(\xii^{-2}),\]
whereas by~\eqref{eq:lambdahighsimple} we get $\lambda_2(\xi)= -\tau^{-1} + \textit{o}(1)$, as $\xii\to\infty$.
\end{example}
\begin{example}
We consider the fourth order operator
\[ Q_2(\partial_t,\partial_x)= \partial_t^2(\partial_t^2-c^2\Delta) + \partial_t(\partial_t^2-\Delta)+\partial_t^2-\Delta,\]
with $c>1$, that is, we fix
\[ P_4(\lambda,\xi')=\lambda^2(\lambda^2-c^2), \quad P_3(\lambda,\xi')=\lambda(\lambda^2-1), \quad P_2(\lambda,\xi')=\lambda^2-1. \]
Then $\lambda_{3,4}(0)=(-1\pm 3i)/2$ and, by Lemma~\ref{lem:low2weak}, we obtain the asymptotic behaviors
\[ \lambda_{1,2}(\xi)= \pm i\xii \mp i\,\frac{c^2-1}{2} \, \xii^3 - \frac{c^2-1}{2} \,\xii^4 + \textit{o}(\xii^4), \]
as $\xi\to0$. Let $\kappa_\pm=(-1\pm i \sqrt{4c^2-1})/(2c^2)$ be the two solutions to~\eqref{eq:kappahigh}, that is, $c^2\kappa^2 + \kappa +1 =0$. By Lemma~\ref{lem:highweak}, we derive the asymptotic behaviors
\[ \lambda_{1,2}(\xi)= \pm ic\xii - \frac{c^2-1}{2c^2} + \textit{o}(1), \qquad \lambda_{3,4}(\xi)=\frac{-1\pm i \sqrt{4c^2-1}}{2c^2} + \textit{o}(1), \]
as $\xii\to\infty$. Applying Theorem~\ref{thm:asympweak}, we get the asymptotic behavior $I_1w$ for $u$, where
\begin{align*}
\mathfrak{F}(I_1w)
    & =\xii^{-1}\hat w(t,\xi)= M\,\frac{e^{i\xii t-i\,\frac{c^2-1}{2} \, \xii^3 t- \frac{c^2-1}{2} \,\xii^4t}-e^{-i\xii t+i\,\frac{c^2-1}{2} \, \xii^3t - \frac{c^2-1}{2} \,\xii^4t}}{2i\xii} \\
    & = M\,\xii^{-1}\sin \Big(\xii t- \frac{c^2-1}{2}\xii^3t\Big)\,e^{-\frac{c^2-1}{2} \,\xii^4t}\,.
\end{align*}
Explicitly,
\[ \|\partial_t^k\big( u(t,\cdot)-I_1w(t,\cdot)\big)\|_{\dot H^s} =\textit{o}\big((1+t)^{-\frac{n}8-\frac{k+s-1}4}\big),\]
for~$n/2+k+s>1$.
\end{example}

\section{Application of Theorem~\ref{thm:CPQ2} to semilinear problems}\label{sec:nonlinear}

The estimates in Theorem~\ref{thm:CPQ2} may be applied to study different nonlinear problems. In this section, we present an application to the problem perturbed by a power nonlinearity. We consider
\begin{equation}\label{eq:CPnl}
\begin{cases}
Q(\partial_t,\partial_x) u =f(D^\alpha u), \quad t\geq 0, x\in\R^n,\\
\partial_t^ju(0,x)=u_j(x),\quad j=0,\ldots,m-1,
\end{cases}
\end{equation}
where $D=(\nabla,\partial_t)$, $\nu=|\alpha|=0,\ldots,m-2$, and $f$ verifies the following local Lipschitz-type condition
\begin{equation}\label{eq:f}
|f(s_2)-f(s_1)|\leq C\,\big(|s_2|+|s_1|\big)^{p-1}\,|s_2-s_1|,
\end{equation}
for some $p>2$. For instance, $f=\pm |\partial_t^{\nu}u|^p$. We also define $\iota=0$ if we consider $Q=Q_1$, and $\iota=1$ if we consider $Q=Q_2$, in~\eqref{eq:CPnl}. Let the space dimension~$n$ be such that
\begin{equation}\label{eq:dimension}
(m-2-\iota-\nu)<n\leq 2(m-1-\iota-\nu).
\end{equation}
Then, assuming small initial data $u_j \in L^1\cap H^{m-1-j}$, $j=0,\ldots,m-1$, global small data solutions to~\eqref{eq:CPnl} exist for
\begin{equation}\label{eq:pcrit}
p> \bar p(n) = 1 + \frac{m-\iota-\nu}{n-(m-2-\iota-\nu)}\,.
\end{equation}
We notice that~\eqref{eq:dimension} and~\eqref{eq:pcrit} imply that $p>2$.
\begin{corollary}\label{cor:nl}
Let $Q=Q_1$ or $Q=Q_2$, as in Theorem~\ref{thm:CPQ2}, and fix $\iota=0$ if $Q=Q_1$ or $\iota=1$ if $Q=Q_2$. Assume that the space dimension~$n$ verifies~\eqref{eq:dimension} and that $p$ in~\eqref{eq:f} verifies~\eqref{eq:pcrit}. Then there exists $\varepsilon>0$ such that if
\[ u_j \in L^1\cap H^{m-1-j}, \quad j=0,\ldots,m-1, \qquad \mathcal A=\sum_{j=0}^{m-1} \big(\|u_j\|_{L^1}+\|u_j\|_{H^{m-1-j}}\big) \leq \varepsilon, \]
there exists a unique solution $u\in X=\cap_{j=0}^{m-1} \mathcal C^j([0,\infty),H^{m-1-j})$ to Cauchy problem~\eqref{eq:CPnl}. Moreover, the solution verifies the decay estimate
\begin{equation}
\label{eq:estnl}\begin{split}
\|\partial_t^ku(t,\cdot)\|_{\dot H^s}
    & \leq C\,(1+t)^{-\frac{n}4-\frac{k+s-(m-2-\iota)}2}\,\mathcal A,
\end{split}\end{equation}
for any $k+s\leq m-1$ such that $n/2+k+s>m-2-\iota$.
\end{corollary}
\begin{proof}
The proof is very classical. The local existence of the solution follows by standard arguments and Sobolev embeddings, so we only prove the existence and uniqueness of the global solution. We prove Corollary~\ref{cor:nl} for $D^\alpha=\partial_t^\nu$, being the proof in the other cases analogous. For a given small $\delta>0$, we fix $\kappa=\delta+m-2-\iota-n/2$ and we define the norm
\begin{align*}
\|v\|_X
    & = \sup_{t\geq0} (1+t)^{\frac{n}4+\frac{\kappa-(m-2-\iota)}2}\,\sum_{k+s=\kappa} \|\partial_t^ku(t,\cdot)\|_{\dot H^s}\\
    & \qquad + \sup_{t\geq0} \sum_{k+s=[\kappa]+1}^{m-1} (1+t)^{\frac{n}4+\frac{k+s-(m-2-\iota)}2}\,\|\partial_t^ku(t,\cdot)\|_{\dot H^s},
\end{align*}
where $[\kappa]=\max\{j\in\N: \ j\leq \kappa\}$. In particular, we choose sufficiently small $\delta>0$ such that $\kappa\leq \nu+n(1/2-1/p)$. We stress that $m-2-\iota-n/2<\nu+n(1/2-1/p)$, thanks to~\eqref{eq:pcrit}.

By Theorem~\ref{thm:CPQ2}, the solution $\bar u$ to~\eqref{eq:CPlin} verifies the estimate $\|\bar u\|_X \leq C_1\mathcal A$, for some $C_1>0$. Moreover, if $\|v\|_X,\|w\|_X<\infty$, then, by~\eqref{eq:f} and H\"older's inequality,
\begin{equation}\label{eq:estFDv}\begin{split}
\|(f(\partial_t^\nu v)-f(\partial_t^\nu w))(s,\cdot)\|_{L^q}
    & \leq C\,\|\partial_t^\nu(v-w)(s,\cdot)\|_{L^{pq}}\big(\|\partial_t^\nu v(s,\cdot)\|_{L^{pq}}^{p-1}+\|\partial_t^\nu w(s,\cdot)\|_{L^{pq}}^{p-1}\big)\\
    & \lesssim (1+s)^{-\frac{n}{2q}\,(p-1)-\frac{\nu-(m-2-\iota)}2\,p}\,\|v-w\|_X\,\big(\|v\|_{X}^{p-1}+\|w\|_{X}^{p-1}\big),
\end{split}\end{equation}
for any $q\geq1$, where we used $\|\cdot\|_{L^{pq}}\lesssim \|\cdot\|_{\dot H^{n\left(\frac12-\frac1{qp}\right)}}$ and the definition of $\|\cdot\|_X$.

We define the integral operator
\[ Fv(t,\cdot) = \int_0^t K(t-s)\ast f(\partial_t^\nu v)(s,\cdot)\,ds, \]
where $K$ denotes the fundamental solution to~\eqref{eq:CPlin}, i.e., $K(t)\ast u_{m-1}$ solves~\eqref{eq:CPlin} with $u_0=\ldots=u_{m-2}=0$. If we prove that
\begin{equation}\label{eq:Fest}
\|Fv-Fw\|_X \leq C_2\,\|v-w\|_X\,\big(\|u\|_{X}^{p-1}+\|w\|_{X}^{p-1}\big),\quad \text{for any $v,w\in X$,}
\end{equation}
then we get the existence and uniqueness of the solution to~\eqref{eq:CPnl}, by Banach's contraction mapping theorem. Indeed, let~$R=2C_1\mathcal A$. Assuming $\mathcal A\leq\varepsilon$, it follows~$CR^{p-1}<1/2$ for sufficiently small $\varepsilon$. Then~$\|\bar u\|_X\leq R/2$ and~$F$ is a contraction on~$X_R= \{u\in X: \ \|u\|_{X}\leq R\}$. The solution~$u$ to~\eqref{eq:CPnl} is the unique fixed point for~$\bar u(t,x)+Fu(t,x)$. Moreover, $\|u\|_X\leq R= 2C_1\,\mathcal A$, so that~\eqref{eq:estnl} follows for any $k+s\leq m-1$ such that $n/2+k+s\geq\delta+m-2-\iota$. Since $\delta>0$ is arbitrarily small, we may obtain~\eqref{eq:estnl} for $n/2+k+s>m-2-\iota$.

To prove~\eqref{eq:Fest}, we apply Theorem~\ref{thm:CPQ2} to $K(t-s)\ast f(\partial_t^\nu v)(s,\cdot)$, with $q=1$ for $s\in[0,t/2]$ and with $q\in[1,2]$ such that
\[ 0<\frac{n}2\left(\frac1q-\frac12\right)+\frac{k+s-(m-2-\iota)}2 \leq 1, \]
for $s\in[t/2,t]$. Using~\eqref{eq:estFDv}, we obtain, for $k+s\geq\kappa$:
\begin{align*}
& \|\partial_t^k(Fv-Fw)(t,\cdot)\|_{\dot H^s} \lesssim (I_1(t)+I_2(t))\,\|v-w\|_X\,\big(\|u\|_{X}^{p-1}+\|w\|_{X}^{p-1}\big), \quad \text{where} \\
& I_1  = \int_0^{t/2} (1+t-s)^{-\frac{n}4-\frac{k+s-(m-2-\iota)}2}\,(1+s)^{-\frac{n}{2}\,(p-1)-\frac{\nu-(m-2-\iota)}2\,p}\,ds\\
& I_2  = \int_{t/2}^t (1+t-s)^{-\frac{n}2\left(\frac1q-\frac12\right)-\frac{k+s-(m-2-\iota)}2}\,(1+s)^{-\frac{n}{2q}\,(p-1)-\frac{\nu-(m-2-\iota)}2\,p}\,ds.
\end{align*}
Noticing that
\[ -\frac{n}{2}\,(p-1)-\frac{\nu-(m-2-\iota)}2\,p<-1,\]
if, and only if, \eqref{eq:pcrit} holds, we may estimate
\begin{align*}
I_1
    & \lesssim (1+t)^{-\frac{n}4-\frac{k+s-(m-2-\iota)}2}\,\int_0^{t/2} (1+s)^{-\frac{n}{2}\,(p-1)-\frac{\nu-(m-2-\iota)}2\,p}\,ds \\
    & \lesssim (1+t)^{-\frac{n}4-\frac{k+s-(m-2-\iota)}2}\,,\\
I_2
    & \lesssim (1+t)^{-\frac{n}{2q}\,(p-1)-\frac{\nu-(m-2-\iota)}2\,p}\,\int_{t/2}^t (1+t-s)^{-\frac{n}2\left(\frac1q-\frac12\right)-\frac{k+s-(m-2-\iota)}2}\,ds\\
    & \lesssim (1+t)^{-\frac{n}{2q}\,(p-1)-\frac{\nu-(m-2-\iota)}2\,p+1-\frac{n}2\left(\frac1q-\frac12\right)-\frac{k+s-(m-2-\iota)}2}\,\log(1+t)\\
    & \leq (1+t)^{-\frac{n}4-\frac{k+s-(m-2-\iota)}2}\,.
\end{align*}
Since we proved~\eqref{eq:Fest}, this concludes the proof.
\end{proof}
%
We may prove the nonexistence of nontrivial global-in-time solutions to~\eqref{eq:CPnl} when~$f=|u|^p$ if $M\geq0$ and $f=-|u|^p$ if $M\leq0$, for any~$p\in(1,\bar p(n)]$, if we add the assumption that $P_{m-1-\iota}(0,\xi)=0$. Indeed, we may prove that even $L^p_\lloc$ solutions do not exist globally-in-time. We say that $u\in L^p_\lloc([0,\infty)\times\R^n)$ is a global-in-time weak solution to~\eqref{eq:CPnl} with~$f=\pm |u|^p$ and $u_j\in L^1_\lloc(\R^n)$, $j=0,\ldots,m-1$, if %
for any $\psi\in\mathcal C_c^\infty([0,\infty))$, with $\psi(t)=1$ in a neighborhood of the origin, and for any $\varphi\in\mathcal C_c^\infty(\R^n)$, the following integral equality holds:
\begin{equation}\label{eq:weakblow}\begin{split}
\pm \int_0^\infty \int_{\R^n} |u(t,x)|^p\,\psi(t)\,\varphi(x)\,dxdt
    & = \int_0^\infty\int_{\R^n} u(t,x)\,Q(-\partial_t,-\partial_x)\big(\psi(t)\,\varphi(x)\big)\,dxdt \\
    & \qquad - \sum_{k=1}^{m} \sum_{k+|\alpha|=m-1-\iota}^m c_{k,\alpha} \int_{\R^n} u_{k-1}(x)\,(-\partial_x)^\alpha\varphi(x)\,dx\,.
\end{split}\end{equation}
Integrating by parts, it is easy to prove that smooth solutions are weak solutions.
\begin{proposition}\label{prop:critical}
Let $Q=Q_1$ or $Q=Q_2$, as in Theorem~\ref{thm:CPQ2}, and fix $\iota=0$ if $Q=Q_1$ or $\iota=1$ if $Q=Q_2$. Assume that $P_{m-1-\iota}(0,\xi)=0$. Assume that $u_j\in L^1(\R^n)$ for any $j=0,\ldots,m-1$. 
Let~$n\leq m-2-\iota$ and $p>1$, or~$1<p\leq \bar p(n)$, where $\bar p(n)=1 + (m-\iota)/(n-(m-2-\iota))$ as in~\eqref{eq:pcrit}. Then there is no global-in-time nontrivial weak solution to~\eqref{eq:CPnl}, in the sense of~\eqref{eq:weakblow}, with $f=|u|^p$ if $M\geq0$ and with $f=-|u|^p$, if $M\leq0$, where $M$ is as in~\eqref{eq:moment}.
\end{proposition}
With the notation in~\cite{DL03}, the assumption $P_{m-1-\iota}(0,\xi)=0$ guarantees that the operator $Q(\partial_t,\partial_x)$ is a sum of quasi-homogeneous operators of type $(h_i,2,1)$, with $\min h_i=m-\iota$, and the proof essentially follows by~\cite[Theorem 4.2]{DL03}, thanks to the assumption on the moment condition. However, for the ease of reading, we provide a self-contained proof of Proposition~\ref{prop:critical}.
\begin{proof}
Assume, by contradiction, that $u$ is a global-in-time nontrivial weak solution. Let $\Psi\in\mathcal C_c^\infty([0,\infty))$, nonincreasing, with $\psi(t)=1$ in~$[0,1]$. For a given integer $\ell\geq mp'$, where $p'=p/(p-1)$, we set $\psi(t)=(\Psi(t))^\ell$ and $\varphi(x)=\psi(|x|)$, then we define $\psi_R(t)=\psi(R^{-2}t)$ and $\varphi_R(x)=\varphi(R^{-1}x)$ for $R\gg1$. If we put
\begin{align*}
I_R
    & =\int_0^\infty \int_{\R^n} |u(t,x)|^p\,\psi_R(t)\,\varphi_R(x)\,dxdt,\\
I_R^*
    & =\int_{\Omega_R^*} |u(t,x)|^p\,\psi_R(t)\,\varphi_R(x)\,dxdt,
\end{align*}
where $\Omega_R^*=[0,\infty)\times\R^n\setminus \big([0,R]\times B_R\big)$, by H\"older and Young inequality, we derive
\begin{align*}
I_R
    & \leq \frac12 I_R^* + C\,\int_0^\infty\int_{\R^n} |Q(-\partial_t,-\partial_x)\big(\psi_R(t)\,\varphi_R(x)\big)|^{p'}\,\big(\psi_R(t)\,\varphi_R(x)\big)^{-\frac{p'}{p}}\,dxdt \\
    & \qquad \mp \sum_{k=1}^{m} \sum_{k+|\alpha|=m-1-\iota}^m c_{k,\alpha} \int_{\R^n} u_{k-1}(x)\,(-\partial_x)^\alpha\varphi_R(x)\,dx\,.
\end{align*}
We stress that the assumption $\ell\geq mp'$ is here used to obtain that
\[ |Q(-\partial_t,-\partial_x)\big(\psi(t)\,\varphi(x)\big)|^p \leq C\,\psi(t)\,\varphi(x), \]
whereas the fact that $I_R^*$ appears, instead of $I_R$, after H\"older and Young inequality, is related to the fact that $Q(-\partial_t,-\partial_x)\big(\psi(t)\,\varphi(x)\big)=0$ in $[0,R]\times B_R$.

Thanks to the homogeneity of the monomials in $Q(\partial_t,\partial_x)$, we easily get:
\begin{align*}
& \int_0^\infty\int_{\R^n} |Q(-\partial_t,-\partial_x)\big(\psi_R(t)\,\varphi_R(x)\big)|^{p'}\,\big(\psi_R(t)\,\varphi_R(x)\big)^{-\frac{p'}{p}}\,dxdt\\
& \qquad \leq C\,R^{-(m-\iota)p'+n+2} \leq C,
\end{align*}
with $C$ independent on~$R$, thanks to the assumption $p\leq \bar p(n)$, that is, $(m-\iota)p'\leq n+2$. On the other hand,
\[ \lim_{R\to\infty} \sum_{k=1}^{m} \sum_{k+|\alpha|=m-1-\iota}^m c_{k,\alpha} \int_{\R^n} u_{k-1}(x)\,(-\partial_x)^\alpha\varphi_R(x)\,dx = M, \]
by Lebesgue dominated convergence theorem, since $\varphi_R(x)\to1$ and $(-\partial_x)^\alpha\varphi_R(x)\to0$ when $\alpha\neq0$, for any $x\in\R^n$, as $R\to\infty$. As a consequence, $I_R$ is bounded, so that, by Beppo-Levi monotone convergence theorem,
\[ \int_0^\infty \int_{\R^n} |u(t,x)|^p\,dxdt = \lim_{R\to\infty} I_R\leq C, \]
that is, $u\in L^p([0,\infty)\times \R^n)$. As a consequence, by Lebesgue dominated convergence theorem, $I_R^*\to0$ as $R\to\infty$, so that we obtain the inequality
\[ \int_0^\infty \int_{\R^n} |u(t,x)|^p\,dxdt \leq \mp M. \]
The contradiction follows, since $\mp M\leq0$, so that $u\equiv0$.
\end{proof}
\begin{remark}
The results in this section may be generalized in several ways. When $p<\bar p(n)$, lifespan estimates for the solution to~\eqref{eq:CPnl} may be easily obtained. By assuming small initial data $u_j \in L^q\cap H^{m-1-j}$, for some $q\in(1,2]$, the critical exponent $\bar p(n)$ is replaced by $\bar p(n/q)$; moreover, in this case, global small data solutions exist for critical and supercritical powers $p\geq\bar p(n/q)$, and the nonexistence result in Proposition~\ref{prop:critical} may be extended to all subcritical powers~$1<p<\bar p(n/q)$, under a suitable sign assumption on the initial data. 
We omit the details for brevity, addressing the reader to~\cite{DAE17NA} and the reference therein. Possibly, critical nonlinearities may be investigated as done in~\cite{EGR} for the damped wave equation.
\end{remark}

%

%

\end{document}